\documentclass[11pt,reqno]{amsart} %
\usepackage{amsmath,amscd}
\usepackage{amsthm}
\usepackage{amssymb}
\usepackage{amsfonts}
\usepackage{latexsym}
\usepackage{url}
\usepackage{latexsym}
\usepackage{graphicx,color,subfig}
\usepackage{enumerate,fancybox}
\usepackage[colorlinks=true, citecolor=blue]{hyperref} 

\DeclareMathOperator\dist{dist}

\newcommand{\capGm}{{\mathrm{cap}}(\Gamma)}

\newtheorem{theorem}{Theorem}[section]

\newtheorem{lemma}{Lemma}[section]
\newtheorem{definition}{Definition}[section]

\newtheorem{remark}{Remark}[section]
\newtheorem{problem}{Problem}[section]

\newtheorem{corollary}{Corollary}[section]

\newtheorem*{mythm}{Theorem}
\newtheorem*{mylemma}{Lemma}

\begin{document}
\title[Strong asymptotics for Bergman polynomials]
{Strong asymptotics for Bergman polynomials over domains with corners and applications}

\date{\today}

\author[N. Stylianopoulos]{Nikos Stylianopoulos}
\address{Department of Mathematics and Statistics,
         University of Cyprus, P.O. Box 20537, 1678 Nicosia, Cyprus}
\email{nikos@ucy.ac.cy}
\urladdr{http://ucy.ac.cy/~nikos}


\begin{abstract}
Let $G$ be a bounded simply-connected domain in the complex plane $\mathbb{C}$,
whose boundary $\Gamma:=\partial G$ is a Jordan curve, and let $\{p_n\}_{n=0}^{\infty}$ denote the sequence of Bergman polynomials of $G$. This is defined as the unique sequence
$$
p_n(z) = \lambda_n z^n+\cdots, \quad \lambda_n>0,\quad n=0,1,2,\ldots,
$$
of polynomials that are orthonormal with respect to the inner product
$$
\langle f,g\rangle := \int_G f(z) \overline{g(z)} dA(z),
$$
where $dA$ stands for the area measure.

We establish the strong asymptotics for $p_n$ and $\lambda_n$, $n\in\mathbb{N}$, under the
assumption that $\Gamma$ is piecewise analytic. This complements an investigation started in 1923 by T.\ Carleman, who
derived the strong asymptotics for $\Gamma$ analytic, and carried over by P.K.\ Suetin in the 1960's, who
established them for smooth $\Gamma$.
In order to do so, we use a new approach based on tools from quasiconformal mapping theory.
The impact of the resulting theory is demonstrated in a number of applications, varying from  coefficient estimates in
the well-known class $\Sigma$ of univalent functions and a connection with operator theory, to the computation of
capacities and a reconstruction algorithm from moments.

\end{abstract}

\maketitle
\allowdisplaybreaks
\section{Introduction and main results}\label{section:intro}
Let $G$ be a bounded simply-connected domain in the complex plane $\mathbb{C}$,
whose boundary $\Gamma:=\partial G$ is a Jordan curve and let
$\{p_n\}_{n=0}^{\infty}$ denote the sequence of  Bergman polynomials of
$G$. This is defined as the unique sequence of polynomials
\begin{equation}\label{eq:pndef}
p_n(z) = \lambda_n z^n+ \cdots, \quad \lambda_n>0,\quad n=0,1,2,\ldots,
\end{equation}
that are orthonormal with respect to the inner product
$$
\langle f,g\rangle_G := \int_G f(z) \overline{g(z)} dA(z),
$$
where $dA$ stands for the area  measure. We denote by $L_a^2(G)$ the Hilbert space of functions $f$
analytic in $G$, for which
$$
\|f\|_{L^2(G)}:=\langle f,f\rangle_G^{1/2}<\infty,
$$
and recall
that the sequence of polynomials $\{p_n\}_{n=0}^\infty$ forms a complete orthonormal system for $L_a^2(G)$.

Let $\Omega:=\overline{\mathbb{C}}\setminus\overline{G}$ denote the complement of $\overline{G}$ and let $\Phi$ denote the conformal map  $\Omega\to\Delta:=\{w:|w|>1\}$, normalized so that near infinity
\begin{equation}\label{eq:Phi}
\Phi(z)=\gamma z+\gamma_0+\frac{\gamma_1}{z}+\frac{\gamma_2}{z^2}+\cdots,\quad \gamma>0.
\end{equation}
Finally, let $\Psi:=\Phi^{-1}:\Delta\to\Omega$ denote the inverse conformal map. Then,
\begin{equation}\label{eq:Psi}
\Psi(w)=bw+b_0+\frac{b_1}{w}+\frac{b_2}{w^2}+\cdots, \quad |w|> 1,
\end{equation}
where $b=1/\gamma$ gives the (\textit{logarithmic}) \textit{capacity} $\textup{cap}(\Gamma)$ of $\Gamma$.

As in the bounded case, we define the inner product
$$
\langle f,g\rangle_\Omega := \int_\Omega f(z) \overline{g(z)} dA(z),
$$
on the unbounded domain $\Omega$ and denote by $L_a^2(\Omega)$  the Hilbert space of functions $f$
analytic in  $\Omega$, for which
$$
\|f\|_{L^2(\Omega)}:=\langle f,f\rangle_\Omega^{1/2}<\infty.
$$
We note that $L_a^2(G)$ and $L_a^2(\Omega)$ are known as the \textit{Bergman spaces} of $G$
and $\Omega$, respectively. It is easy to see that for $f\in L_a^2(\Omega)$ to hold it is necessary $f(z)$ has around infinity a Laurent series expansion starting with $1/z^2$.

The main purpose of the paper is to establish the strong asymptotics of the leading coefficients
$\{\lambda_n\}_{n\in\mathbb{N}}$ and the Bergman polynomials $\{p_n\}_{n\in\mathbb{N}}$, in $\Omega$, for non-smooth
boundary $\Gamma$. We do this under the assumption that $\Gamma$ is \textit{piecewise analytic without cusps}. This means
that $\Gamma$ consists of a finite set of analytic arcs that meet at exterior angles $\omega\pi$, with $0<\omega<2$.
Thus, we allow $\Gamma$ to have corners. In this sense, our results complement an investigation started
by T.~Carleman \cite{Ca23} in 1923, who derived the strong asymptotics under the assumption that $\Gamma$ is analytic,
and was carried over by P.K.~Suetin \cite{Su74} in the 1960's, who verified them for smooth $\Gamma$.

The techniques employed in both \cite{Ca23} and \cite{Su74} are tied to the specific properties that characterize the
mapping functions $\Phi$ and $\Psi$ in cases when $\Gamma$ is analytic, or smooth, and it turns out that they are not
suitable for treating domains with corner. In order to overcome this we have developed an approach, which we believe
to be novel. This approach involves, in particular, new techniques from the theory of quasiconformal mapping and a new
sharp estimate, concerning the growth of a polynomial in terms of its $L^2$-norm.

Our main results are the following three theorems.
\begin{theorem}\label{thm:finelambdan}
Assume that $\Gamma$ is piecewise analytic without cusps. Then, for any $n\in\mathbb{N}$, it holds that
\begin{equation}\label{eqinthm:finelambdan}
{\frac{n+1}{\pi}\frac{\gamma^{2(n+1)}}{\lambda_n^2}=1-\alpha_n,}
\end{equation}
where
\begin{equation}\label{eqinthm:finelambdanii}
0\le\alpha_n\le c_1(\Gamma)\,\frac{1}{n}.
\end{equation}
\end{theorem}

\begin{theorem}\label{thm:finepn}
Under the assumptions of Theorem \ref{thm:finelambdan}, for any $n\in\mathbb{N}$, it holds that
\begin{equation}\label{eqinthm:finepn}
{p_n(z)=\sqrt{\frac{n+1}{\pi}}\,\Phi^n(z)\Phi^\prime(z)
\left\{1+A_n(z)\right\}},\quad z\in\Omega,
\end{equation}
where
\begin{eqnarray} \label{eqinthm:finepnii1}
|A_n(z)|\le \frac{c_2(\Gamma)}{\dist(z,\Gamma)\,|\Phi^\prime(z)|}\,\frac{1}{\sqrt{n}}
+c_3(\Gamma)\,\frac{1}{n} .
\end{eqnarray}
\end{theorem}
Above and in the sequel we use $c(\Gamma)$, $c_1(\Gamma)$, $c_2(\Gamma)$, e.t.c., to denote non-negative  constants that depend only on $\Gamma$. We also use $\dist(z,B)$ to denote the (Euclidian) distance of $z$ from a set $B$ and call the quantities $\alpha_n$ and $A_n(z)$, defined by (\ref{eqinthm:finelambdan}) and (\ref{eqinthm:finepn}), as the \textit{strong asymptotic errors} associated with $\lambda_n$ and $p_n(z)$, respectively.

From (\ref{eqinthm:finepnii1}) and the well-known distortion  property of conformal mappings
\begin{equation}\label{eq:distortion}
\dist(\Phi(z),\partial\mathbb{D})\le 4\,\dist(z,\Gamma)\,|\Phi^\prime(z)|,\quad z\in\Omega;
\end{equation}
see, e.g., \cite[p.~23]{ABbook}, we arrive at another estimate for $A_n(z)$, which does not involve the derivative of $\Phi$, i.e.,
\begin{eqnarray} \label{eqinthm:finepnii2}
|A_n(z)|\le\frac{c_4(\Gamma)}{|\Phi(z)|-1}\,\frac{1}{\sqrt{n}}
+c_3(\Gamma)\,\frac{1}{n}, \quad z\in\Omega.
\end{eqnarray}

Our next result provides an interesting link between the  Bergman polynomials and the problem of coefficient estimates in Univalent Functions Theory. This result is established under the assumption that $\Gamma$
belongs to a broader class of Jordan curves than the one appearing in Theorem~\ref{thm:finelambdan}, namely the class of
quasiconformal curves. We recall that a Jordan curve $\Gamma$ is \textit{quasiconformal} if there exists a constant $M$ such
that,
$$
\textup{diam}\Gamma(z,\zeta)\le M |z-\zeta|,\mbox{ for all }\, z,\zeta\in\Gamma,
$$
where
$\Gamma(z,\zeta)$ is the arc (of smaller diameter) of $\Gamma$ between $z$ and $\zeta$.
In connection with the assumptions of Theorem~\ref{thm:finelambdan}, we also recall that a piecewise analytic Jordan
curve is quasiconformal if and only if has no cusps. 
The assumption that $\Gamma$ is quasiconformal ensures the existence of an associated $K$-quasiconformal reflection
$y(z)$, for some $K\ge 1$, which is characterized by the properties (A1)--(A3) stated in
Remark~\ref{lem:prop-qc} below. The existence of the quasiconformal reflection was established by Ahlfors in \cite{Ah63}.
All our estimates that are derived under the assumption that $\Gamma$ is quasiconformal are
given in terms of the constant
\begin{equation}\label{eq:refcoef}
k:=(K-1)/(K+1),
\end{equation}
which in the sequel we refer to  as the \textit{reflection factor of $\Gamma$ (associated with $y$)}.
We note that $0\le k<1$, with $k=0$ if $\Gamma$ is a circle.

In the next theorem we require, in addition, that  $\Gamma$ is rectifiable. Note that there are examples of
non-rectifiable quasiconformal curves; see, e.g., \cite[p.~104]{LV}.
However, any quasiconformal curve has zero area.

Our result shows that the strong asymptotic error $\alpha_n$ cannot decay faster than $(n+1)|b_{n+1}|^2$, where
$b_{n+1}$ is the coefficient of $1/w^{n+1}$ in the Laurent series expansion (\ref{eq:Psi}) of $\Psi(w)$.
\begin{theorem}\label{thm:alphange}
Assume that $\Gamma$ is quasiconformal and rectifiable. Then, for any $n\in\mathbb{N}$, it holds that
\begin{equation}\label{eq:alphange}
\alpha_{n}\ge\,\frac{\pi\,(1-k^2)}{A(G)}\,(n+1)\,|b_{n+1}|^2,
\end{equation}
where $A(G)$ denotes the area of $G$ and $k$ is the reflection factor of $\Gamma$.
\end{theorem}
As it was noted above, Theorem~\ref{thm:alphange} provides a link between the problems of estimating the error in the
strong asymptotics for $\lambda_n$ and of estimating coefficients in the well-known class $\Sigma$, consisting  of
functions analytic and univalent in $\Delta\setminus\{\infty\}$ that have a Laurent series expansion of the form
(\ref{eq:Psi}) with $b=1$. This latter problem is one of the best-studied in Geometric Function Theory; see, e.g.,
\cite{Po75} and \cite{Du83}.

In another application, the result of Theorem~\ref{thm:alphange} can be used to  discuss the sharpness in the decay
of order $O(1/n)$, predicted for the sequence $\{\alpha_n\}$ by Theorem~\ref{thm:finelambdan}. Ideally, in order to
show by means of Theorem~\ref{thm:alphange} that the estimate (\ref{eqinthm:finelambdanii}) is sharp, it would
suffice to find a domain $G$, bounded by a piecewise analytic, and without cups, Jordan curve $\Gamma$, for which
it would hold $|b_n|\ge c_4/n$.
In view of the estimate $|b_n|\le c_5/n$, for $n\in\mathbb{N}$, however, which was obtained by Gaier in \cite{Ga99} for
$\Gamma$ piecewise Dini-smooth, this seems already tricky, because if $\Gamma$ is piecewise analytic (even with cusps)
then is also piecewise Dini-smooth.
Moreover, in view of the estimate $|b_n|\le c_4/n^{1+\omega}$, $n\in\mathbb{N}$, with $0<\omega<2$,
which is established in Section~\ref{subsec:coef-esti} below, the use of Theorem~\ref{thm:alphange} for proving
this kind of sharpness is of no help.

Nevertheless, Theorem~\ref{thm:alphange} can be employed to show that the order $O(1/n)$ in (\ref{eqinthm:finelambdanii})
is best possible in a different sense:
\begin{remark}\label{rem:alphan-low-esti}
For any $\epsilon>0$, there exists a domain $G$, which is bounded by a piecewise analytic Jordan curve $\Gamma$,
such that for the associated strong asymptotic error $\alpha_n$ there holds
\begin{equation}\label{eq:alphan-low-esti}
\alpha_n \ge c_5(\Gamma)\frac{1}{n^{1+\epsilon}},
\end{equation}
for some positive constant $c_5(\Gamma)$ and infinitely many $n\in\mathbb{N}$.
\end{remark}
This will be shown in a forthcoming article with the help of a domain $G$ whose boundary consists of two symmetric,
with respect to the imaginary axis,  circular arcs that meet at $i$ and $-i$, forming exterior angles $\pi/N$, with
$N\in\mathbb{N}\setminus\{1\}$. More precisely, in this case it can be shown that
$$
|b_{2n+1}|\asymp \frac{1}{n^{1+1/N}},\quad n\in\mathbb{N},
$$
which, in view of Theorem~\ref{thm:alphange}, implies (\ref{eq:alphan-low-esti}).

In addition to the above, we present two examples and certain numerical evidence supporting the hypothesis that the order
$O(1/n)$ in (\ref{eqinthm:finelambdanii}) is, indeed, sharp.

The first example is based on a Jordan curve constructed by Clunie in \cite{Cl59}, for which the sequence $\{n\,b_n\}_{n\in\mathbb{N}}$ is unbounded. More precisely, for $\epsilon=1/50$, there exists some subsequence
$\mathcal{N}$ of $\mathbb{N}$, such that
\begin{equation}\label{eq:clunie}
n|b_n|> n^\epsilon,\quad n\in\mathcal{N}.
\end{equation}
It was shown by Gaier in \cite[\S~4.2]{Ga99} that Clunie's curve is, eventually, quasiconformal.

The second example is generated by the function
\begin{equation}\label{eq:Psi-hypom}
\Psi(w)=w+\frac{1}{(m-1)w^{m-1}},\quad |w|\ge 1,
\end{equation}
For any $m\ge 3$, this function maps $\Delta$ conformally onto the exterior of a symmetric $m$-cusped hypocycloid $H_m$,
which is a piecewise analytic Jordan curve with all exterior angles equal to $2\pi$, and thus not a quasiconformal curve.
Nevertheless, for each $n\ge 2$, $H_{n+1}$ provides an example of a cusped Jordan curve, where $b_n=1/n$.

Regarding numerical evidence, we consider the case where $G$
is the unit half-disk and display in Table~\ref{tab:an-decay} a range of computed values of $\alpha_n$, for
$n=51,\ldots,60$.
These were obtained from the exact value $\gamma=1/\textup{cap}(\Gamma)=3\sqrt{3}/4$ and
the computed values of $\lambda_n$, after constructing in finite high precision and in the way
indicated in Section~\ref{subsec:ArnoldiGS} below, the Bergman polynomials up to degree $60$.
Thus, we expect all the figures quoted in the table to be correct.
The reported values of $\alpha_n$ indicate that the strong asymptotic error for the leading
coefficient decays monotonically to zero.
In view of the estimate
(\ref{eqinthm:finelambdanii}), we test the hypothesis ${\alpha_n\approx{1}/{n^s}}$.
The reported values of $s$ in the table indicate clearly that $s=1$.

Exactly the same behaviour was observed in a number of different non-smooth cases, involving various angles and
shapes. Based on such evidence, we have conjectured the strong asymptotics for non-smooth domains
in \cite[pp.~520--521]{NS06}.

\begin{table}[h]
\begin{tabular}{|c|c|c|}
\hline $ $
$n$& $\alpha_n$           & $s$   ${\vphantom{\sum^{\sum^{\sum^N}}}}$ \\*[3pt] \hline
51 & 0.003\,263\,458\,678 &  -       \\
52 & 0.003\,200\,769\,764 & 0.998\,887     \\
53 & 0.003\,140\,444\,435 & 0.998\,899     \\
54 & 0.003\,082\,351\,464 & 0.998\,911     \\
55 & 0.003\,026\,369\,160 & 0.998\,923     \\
56 & 0.002\,972\,384\,524 & 0.998\,934     \\
57 & 0.002\,920\,292\,482 & 0.998\,946    \\
58 & 0.002\,869\,952\,027 & 0.998\,957    \\
59 & 0.002\,821\,401\,485 & 0.998\,968     \\
60 & 0.002\,774\,426\,207 & 0.998\,979   \\ \hline
\end{tabular}

\medskip
\caption{The rate of decay of $\alpha_n$ for the unit half-disk.} \label{tab:an-decay}
\end{table}

The first ever result regarding strong asymptotics for $\{\lambda_n\}_{n\in\mathbb{N}}$ and $\{p_n\}_{n\in\mathbb{N}}$ was derived by Carleman in~\cite{Ca23}, for domains bounded by analytic Jordan curves. In this case the conformal map $\Phi$ has an analytic and one-to-one continuation across $\Gamma$ inside $G$.
\begin{mythm}[Carleman~\cite{Ca23}; see also \cite{Gabook87}, pp. 12--14]\label{thm:carleman}
Let $L_R$ to denote the level curve  $\{z:|\Phi(z)|=R\}$ and assume that $\rho<1 $ is the smallest number for which $\Phi$ is conformal in the exterior of $L_\rho$. Then, for any $n\in\mathbb{N}$,
\begin{equation}\label{eq:carlman1}
0\le \alpha_n\le c_6(\Gamma)\, \rho^{2n}
\end{equation}
and
\begin{equation}\label{eq:carlman2}
|A_n(z)|\le c_7(\Gamma)\sqrt{n}\,\rho^n,\quad z\in\overline{\Omega}.
\end{equation}
\end{mythm}
The next major step in removing the analyticity assumption on $\Gamma$ was taken by P.K.\ Suetin in the 1960's. For his
results, Suetin requires that the boundary curve $\Gamma$ belongs to the smoothness class $C(p,\alpha)$.
This means that
$\Gamma$ is defined by $z=g(s)$, where $s$ denotes arclength, with $g^{(p)}\in \textup{Lip}\,\alpha$, for some
$p\in \mathbb{N}$ and $0<\alpha<1$. In this case both $\Phi$ and $\Psi$ are $p$ times continuously
differentiable in $\overline{\Omega}\setminus\{\infty\}$ and $\overline{\Delta}\setminus\{\infty\}$, respectively, with
$\Phi^{(p)}$ and $\Psi^{(p)}$ in $\textup{Lip}\,\alpha$. A typical result goes as follows:
\begin{mythm}[Suetin~\cite{Su74}, Thms 1.1 \& 1.2]\label{thm:suetin}
Assume that $\Gamma\in C(p+1,\alpha)$, with $p+\alpha>1/2$. Then, for any $n\in\mathbb{N}$,
\begin{equation}\label{eq:suetin1}
0\le \alpha_n\le c_8(\Gamma)\,\frac{1}{n^{2(p+\alpha)}}
\end{equation}
and
\begin{equation}\label{eq:suetin2}
|A_n(z)|\le c_9(\Gamma)\,\frac{\log n}{n^{p+\alpha}},\quad z\in\overline{\Omega}.
\end{equation}
\end{mythm}

The results of Carleman and Suetin given above, in conjunction with Theorem~\ref{thm:alphange}, yield immediately
estimates for the decay of the coefficients $b_n$, depending on the degree of analyticity, or smoothness, of $\Gamma$.
More precisely:
\begin{corollary}
Under the assumptions of the theorem of Carleman it holds, for any $n\in\mathbb{N}$, that
$$
|b_n|\le c_{10}(\Gamma)\,\frac{\varrho^n}{\sqrt{n}}.
$$
Under the assumptions of the theorem of Suetin, it holds, for any $n\in\mathbb{N}$, that
$$
|b_n|\le c_{11}(\Gamma)\,\frac{1}{n^{p+\alpha+1/2}}.
$$
\end{corollary}

Strong asymptotics for $\lambda_n$ and $p_n$ were also derived by E.R.\ Johnston in his Ph.D.\ thesis \cite{Jo}. These
asymptotics, however, were established under analytic assumptions on certain functions related with the conformal maps
$\Phi$ and $\Psi$ (as compared to the geometric assumptions on $\Gamma$ in the theorems above) and they
do not provide the order of decay of the associated errors. An account of Johnston's results can be found in \cite{RW}.

In addition, we cite the following representative works about strong asymptotics for complex orthogonal
polynomials generated by measures supported on 2-dimensional subsets of $\mathbb{C}$:  (a) Szeg\H{o}'s book
\cite[Ch.\ XVI]{Sz}, for orthogonal polynomials with respect to the arclength measure (the so-called Szeg\H{o}
polynomials) on analytic Jordan curves; (b) Suetin's paper \cite{Su66b}, for weighted Szeg\H{o} polynomials on smooth
Jordan curves; (c) Widom's paper \cite{Wi69} for weighted Szeg\H{o} polynomials on systems of smooth Jordan curves and
smooth Jordan arcs; (d) the recent paper \cite{GPSS} by Gustafsson, Putinar, Saff and the author, for Bergman polynomials
on systems of smooth Jordan domains.

The above list is by no means complete. Nevertheless, we haven't been able to trace in the literature a single result
establishing strong asymptotics for orthogonal polynomials defined by measures supported on non-smooth domains, curves
or arcs. In this connection, we note that the well-known approach that combines the Riemann-Hilbert reformulation of
orthogonal polynomials of Fokas, Its and Kitaev \cite{FIK91}--\cite{FIK92}, with the method of steepest descent, introduced by
Deift and Zhou \cite{DZ}, cannot be applied, at least in its present form, to derive strong asymptotics for Bergman
polynomials associated with non-trivial domains. This is so, because this approach produces, invariably, orthogonal
polynomials that satisfy a finite-term recurrence relation and this is not the case with
the Bergman polynomials, as Theorem~\ref{thm:ftrr} below shows.

By contrast, strong asymptotics for orthogonal polynomials on the real line and the unit circle is a well-studied
subject. From the vast bibliography available, we cite the two volumes of B.\ Simon \cite{SimBoI}--\cite{SimBoII},
which contain a comprehensive treatment of the classical and the spectral theory of orthogonal polynomials on the unit
circle, and the recent breakthrough paper of Lubinsky~\cite{Lu09}, on universality limits for kernel polynomials
defined by positive Borel measures in $(-1,1)$.

The paper is organized as follows:
In Section~\ref{sec:faber} we study the properties of associated Faber polynomials and derive a number of preliminary
results under the assumptions: (a) $G$ is a bounded domain and (b) $\Gamma$ is a rectifiable Jordan curve.
In addition, we state a number of results that needed in the proofs of the three main theorems, under
increasing assumptions on $\Gamma$, namely:
(c) $\Gamma$ is a quasiconformal curve and (d) $\Gamma$ is a piecewise analytic curve.
The main result of Section~\ref{sec:poly-est} is a sharp estimate that relates the growth of a polynomial
in $\Omega$ to its $L^2$-norm in $G$. This estimate is essential for establishing Theorem~\ref{thm:finepn}.
Sections~\ref{proofs-qc} and \ref{proofs-pa} are devoted to the proofs of the results stated in
Section~\ref{sec:faber}, regarding assumptions (c) and (d) respectively.
Section~\ref{proofs-main} contain the proofs of the three main theorems of Section~\ref{section:intro}.
Finally, in Section~\ref{sec:appl} we present, in briefly, a number of applications of the strong asymptotics
and the associated theory.

Theorems~\ref{thm:finelambdan}--\ref{thm:finepn}, along with Corollaries~\ref{cor:ratioln}--\ref{cor:ratiopn}
and Theorems~\ref{thn:zeros}, \ref{thm:ftrr}--\ref{thm:DP}  have been presented,
without proofs, in \cite{St-CR10}.

\section{Preliminary results}\label{sec:faber}
\setcounter{equation}{0}
The Faber polynomials $\{F_n\}_{n=0}^\infty$ of $G$ are defined as the polynomial part of the expansion of $\Phi^n(z)$,
 near infinity. Therefore, from (\ref{eq:Phi}),
\begin{equation}\label{eq:PhinFn}
\Phi^n(z)=F_n(z)-E_n(z),\quad z\in\Omega,
\end{equation}
where
\begin{equation}\label{eq:Fndef}
F_n(z)=\gamma^{n}z^n+\cdots,
\end{equation}
is the Faber polynomial of degree $n$ and
\begin{equation}\label{eq:Endef}
E_n(z)=\frac{c^{(n)}_1}{z}+\frac{c^{(n)}_2}{z^2}+\frac{c^{(n)}_3}{z^3}+\cdots,
\end{equation}
is the singular part of $\Phi^n(z)$. According to the asymptotics established by Carleman, the Bergman polynomial
$p_n(z)$ is related to  $\Phi^n(z)\Phi^\prime(z)$. Consequently, we consider the polynomial part of
$\Phi^n(z)\Phi^\prime(z)$, and we denote the
resulting series by $\{G_n\}_{n=0}^\infty$. $G_n(z)$ is the so-called Faber polynomial of the 2nd kind (of degree $n$)
and satisfies
\begin{equation}\label{eq:PhinPhipGn}
\Phi^n(z)\Phi^\prime(z)=G_n(z)-H_n(z),\quad z\in\Omega,
\end{equation}
with
\begin{equation}\label{eq:Gndef}
G_n(z)=\gamma^{n+1}z^n+\cdots,
\end{equation}
and
\begin{equation}\label{eq:Hndef}
H_n(z)=\frac{a^{(n)}_2}{z^2}+\frac{a^{(n)}_3}{z^3}+\frac{a^{(n)}_4}{z^4}+\cdots,
\end{equation}
valid in a neighborhood of infinity.
It follows immediately from (\ref{eq:PhinFn}) and (\ref{eq:PhinPhipGn}) that
\begin{equation}\label{eq:GnFn+1}
G_n(z)=\frac{F_{n+1}^\prime(z)}{n+1}\quad\textup{and}\quad H_n(z)=\frac{E_{n+1}^\prime(z)}{n+1}.
\end{equation}
\begin{lemma}\label{lem:Hn-in-L2}
For any $n\in\mathbb{N}$ it holds that $H_n\in L_a^2(\Omega)$.
\end{lemma}
\begin{proof}
First we observe that the function $\Phi^n(z)\Phi^\prime(z)$ is square integrable in the bounded doubly-connected
domain $D_R$, defined by the boundary curve $\Gamma$ and the level line
\begin{equation*}
L_R:=\{z:|\Phi(z)|=R\}=\{\Psi(w):|w|=R\},\quad R>1.
\end{equation*}
Indeed, by making the change of variables $w=\Phi(z)$, we have
\begin{equation*}
\int_{D_R}|\Phi^n(z)\Phi^\prime(z)|^2dA(z)=\int_{1<|w|<R}|w|^{2n}dA(w)=
\frac{\pi}{n+1}\{R^{2(n+1)}-1\}.
\end{equation*}
Therefore,
\begin{align}\label{eq:intHn1}
\left[\int_{D_R}|H_n(z)|^2dA(z)\right]^{1/2}
&\le\left[\int_{D_R}|\Phi^n(z)\Phi^\prime(z)|^2dA(z)\right]^{1/2}+
\left[\int_{D_R}|G_n(z)|^2dA(z)\right]^{1/2}\nonumber\\
&<\infty.
\end{align}

Next, from the splitting (\ref{eq:PhinPhipGn}) we see that $H_n(z)$ is analytic in $\Omega$ and has a double zero
at infinity.
Assume that (\ref{eq:Hndef}) is valid for $|z|>R_1$.
Then $\limsup_{k\to\infty}|a_k^{(n)}|^{1/k}=R_1$ and, hence, the estimate
\begin{equation*}\label{eq:ank-esti}
|a_k^{(n)}|\le c\, R_2^k
\end{equation*}
holds for some $R_2>R_1$. Therefore, for any $R_3>1$, with $R_3>R_2$, we have:
\begin{align}\label{eq:intHn2}
\int_{|z|>R_3}|H_n(z)|^2dA(z)
&=\int_0^{2\pi}\int_{R_3}^\infty|H_n(r\textup{e}^{i\theta})|^2rdrd\theta=
\sum_{k=2}^\infty\frac{|a_k^{(n)}|^2}{(k-1)R_3^{2(k-1)}}\nonumber\\
&\le c \sum_{k=2}^\infty\frac{R_2^{2k}}{(k-1)R_3^{2(k-1)}}<\infty.
\end{align}

Now, choose $R$ sufficiently large so that $D_R$ contains the circle $\{z:|z|=R_3\}$. Then, the
result $\|H_n\|_{L^2(\Omega)}<\infty$ follows at once from the two estimates (\ref{eq:intHn1}) and (\ref{eq:intHn2}).
\end{proof}

Remark~\ref{rem:beta-eps} and Theorem~\ref{thm:epsn} below show that a lot more can be said about the behaviour of
$\|H_n\|_{L^2(\Omega)}$, under additional assumptions on $\Gamma$.

\subsection{Results for rectifiable boundary}\label{sec2:rect}
We assume now that the boundary $\Gamma$ is \textit{rectifiable}. (Further assumption on $\Gamma$ will be
imposed in various parts of the paper.) For rectifiable $\Gamma$, Cauchy's integral formula yields the following
representation for the Faber polynomial $F_n(z)$ and its corresponding singular part $E_n$(z):
\begin{equation}\label{eq:FnIntRep}
F_n(z)=\frac{1}{2\pi i}\int_\Gamma\frac{\Phi^n(\zeta)}{\zeta-z}\,d\zeta,\quad z\in G,
\end{equation}
\begin{equation}\label{eq:EnIntRep}
E_n(z)=\frac{1}{2\pi i}\int_\Gamma\frac{\Phi^n(\zeta)}{\zeta-z}\,d\zeta,\quad z\in \Omega.
\end{equation}

It is well-known that the assumption on $\Gamma$ implies the facts that $\Phi^\prime$ belongs to the Smirnov class
$E^1(\Omega)$, that both $\Phi^\prime$ and $\Psi^\prime$ have non-tangential limits almost everywhere on $\Gamma$ and
$\partial\mathbb{D}$, respectively, and that they are integrable with respect to the arclength measure, i.e.,
\begin{equation}\label{eq:Phi-prime-in-E1}
\int_\Gamma |\Phi^\prime(\zeta)|\,|d\zeta|<\infty,\quad \textup{and}\quad
\int_\mathbb{T} |\Psi^\prime(t)|\,|dt|<\infty;
\end{equation}
see, e.g., \cite[Ch.~10]{Du70},  \cite{KhD82} and \cite[\S 6.3]{Po92}. Hence
$H_n\in E^1(\Omega)$ and therefore (\ref{eq:PhinPhipGn}) yields the following two Cauchy representations:
\begin{equation}\label{eq:GnIntRep}
G_n(z)=\frac{1}{2\pi i}\int_\Gamma\frac{\Phi^n(\zeta)\Phi^\prime(\zeta)}{\zeta-z}\,d\zeta,\quad z\in G,
\end{equation}
and
\begin{equation}\label{eq:HnIntRep}
H_n(z)=\frac{1}{2\pi i}\int_\Gamma\frac{\Phi^n(\zeta)\Phi^\prime(\zeta)}{\zeta-z}\,d\zeta,\quad z\in \Omega;
\end{equation}
cf.\ \cite[Thm 10.4]{Du70}. We note the following estimate, which is a simple consequence
of (\ref{eq:Phi-prime-in-E1}) and the
representation (\ref{eq:HnIntRep}):
\begin{equation}\label{eq:Hn-decay-rect}
|H_n(z)|\le\frac{c_1(\Gamma)}{\dist(z,\Gamma)},\quad z\in\Omega.
\end{equation}

Next, we single out three identities, which we are going to use below.
\begin{lemma}\label{lem:usefull-1}
Assume that the boundary $\Gamma$ is rectifiable. Then, for any $m,n\in\mathbb{N}$, there holds:
\begin{equation}\label{eq:PhimPhin}
\frac{1}{2\pi i}\int_\Gamma \Phi^m(z)\Phi^\prime(z)\overline{\Phi^{n+1}(z)}dz=\delta_{m,n},
\end{equation}
and
\begin{equation}\label{eq:PhiEmHn}
\int_\Gamma H_m(z)\overline{\Phi^{n+1}(z)}dz=0=
\int_\Gamma \Phi^m(z)\Phi^\prime(z)\overline{E_{n+1}(z)}dz,
\end{equation}
where $\delta_{m,n}$ denotes Kronecker's delta function.
\end{lemma}
\begin{proof}
Since $\Phi^\prime\in E^1(\Omega)$, the application of Cauchy's theorem and
the change of variables $w=\Phi(z)$ give, for some $R>1$,
\begin{equation*}
\begin{alignedat}{1}
\frac{1}{2\pi i}\int_\Gamma \Phi^m(z)\Phi^\prime(z)\overline{\Phi^{n+1}(z)}dz
&=\frac{1}{2\pi i}\int_\Gamma \frac{\Phi^m(z)\Phi^\prime(z)}{\Phi^{n+1}(z)}dz\\
&=\frac{1}{2\pi i}\int_{L_R} \frac{\Phi^m(z)\Phi^\prime(z)}{\Phi^{n+1}(z)}dz
=\frac{1}{2\pi i}\int_{|w|=R} \frac{w^m\,dw}{w^{n+1}},
\end{alignedat}
\end{equation*}
and the result (\ref{eq:PhimPhin}) follows from the residue theorem.

Next, using the splitting (\ref{eq:PhinPhipGn}) in conjunction with (\ref{eq:PhimPhin}) we obtain:
\begin{equation*}
\begin{alignedat}{1}
\frac{1}{2\pi i}\int_\Gamma H_m(z)\overline{\Phi^{n+1}(z)}dz
&=\frac{1}{2\pi i}\int_\Gamma G_{m}(z) \overline{\Phi^{n+1}(z)}dz\\
&-\frac{1}{2\pi i}\int_\Gamma \Phi^m(z)\Phi^\prime(z) \overline{\Phi^{n+1}(z)}dz\\
&=\frac{1}{2\pi i}\int_\Gamma \frac{G_{m}(z)}{\Phi^{n+1}(z)}dz-\delta_{m,n}.
\end{alignedat}
\end{equation*}
The first identity in (\ref{eq:PhiEmHn}) then follows from the residue theorem, because the value of the last integral
is $\delta_{m,n}$; cf.\ (\ref{eq:Phi}) and (\ref{eq:Gndef}).

Finally, using the splitting (\ref{eq:PhinFn}), and making
the change of variables $w=\Phi(z)$, we obtain from (\ref{eq:PhimPhin}) that
\begin{equation*}
\begin{alignedat}{1}
\frac{1}{2\pi i}\int_\Gamma \Phi^m(z)\Phi^\prime(z)\overline{E_{n+1}}(z)\,dz
&=\frac{1}{2\pi i}\int_\Gamma \Phi^m(z)\Phi^\prime(z) \overline{F_{n+1}(z)}dz\\&
-\frac{1}{2\pi i}\int_\Gamma \Phi^m(z)\Phi^\prime(z) \overline{\Phi^{n+1}(z)}dz\\
&=\frac{1}{2\pi i}\int_{|w|=1} w^m\overline{F_{n+1}(\Psi(w))}\,dw-\delta_{m,n}.
\end{alignedat}
\end{equation*}
The second identity in (\ref{eq:PhiEmHn}) then follows by means of the residue theorem, which again implies that
the value of the last integral is $\delta_{m,n}$. This is readily verified by noting that (\ref{eq:Psi}) and (\ref{eq:Fndef}) give, for $|w|=1$,
$$
\overline{F_{n+1}(\Psi(w))}=\overline{\gamma^{n+1}b^{n+1}w^{n+1}+\cdots}=\overline{w^{n+1}+\cdots}=1/w^{n+1}+\cdots.
$$
\end{proof}

With the help of $G_n(z)$ we define an auxiliary polynomial that plays a crucial role in the course our study:
\begin{equation}\label{eq:qndef}
q_{n-1}(z):=G_n(z)-\frac{\gamma^{n+1}}{\lambda_n}p_n(z),\quad n\in\mathbb{N}.
\end{equation}
Observe that $q_{n-1}(z)$ has degree at most $n-1$, but it can be identical to zero, as the special
case $G\equiv\mathbb{D}$ shows.

By noting the relation
\begin{equation}\label{eq:pnPhinPhip}
p_n(z)=\frac{\lambda_n}{\gamma^{n+1}}\,\Phi^n(z)\Phi^\prime(z)
\left\{1+\frac{H_n(z)}{\Phi^n(z)\Phi^\prime(z)}
-\frac{q_{n-1}(z)}{\Phi^n(z)\Phi^\prime(z)}\right\},
\end{equation}
which follows at once from (\ref{eq:PhinPhipGn}) and (\ref{eq:qndef}) and is valid for any $z\in\Omega$ (since
$\Phi^\prime(z)\ne 0$), it is not surprising that we formulate our results in terms of the  following two
sequences of nonnegative numbers:
\begin{equation}\label{eq:betandef}
\beta_n:=\frac{n+1}{\pi}\|q_{n-1}\|^2_{L^2(G)},\quad n\in\mathbb{N},
\end{equation}
and
\begin{equation}\label{eq:epsndef}
\varepsilon_{n}:=\frac{n+1}{\pi}\|H_n\|^2_{L^2(\Omega)},\quad n\in\mathbb{N}.
\end{equation}

The proof of  Theorems~\ref{thm:finelambdan} and \ref{thm:finepn} amounts, eventually, to establishing that the
two sequences $\{\beta_n\}_{n\in\mathbb{N}}$ and $\{\varepsilon_n\}_{n\in\mathbb{N}}$ decay to zero like $O(1/n)$.
To this end, a representation  of $\beta_n$ and $\varepsilon_n$ as line integrals will be useful:

\begin{lemma}\label{lem-usefull-2}
Assume that the boundary $\Gamma$ is rectifiable. Then, for any $n\in\mathbb{N}$, there holds:
\begin{equation}\label{eq:betanEn}
\beta_n=\frac{1}{2\pi i}\,\int_\Gamma q_{n-1}(z)\overline{E_{n+1}(z)}\,dz,
\end{equation}
and
\begin{equation}\label{eq:epsnEn}
\varepsilon_{n}=-\frac{1}{2\pi i}\,\int_\Gamma H_n(z)\overline{E_{n+1}(z)}\,dz.
\end{equation}
\end{lemma}
\begin{proof}
To derive (\ref{eq:betanEn}) we use the orthogonality of $p_n$, (\ref{eq:GnFn+1}) and Green's formula to conclude, in steps, that
\begin{equation*}
\begin{alignedat}{1}
\|q_{n-1}\|^2_{L^2(G)}
&=\langle q_{n-1},G_n-\frac{\gamma^{n+1}}{\lambda_n}p_n\rangle=\langle q_{n-1},G_n\rangle \\
&=\int_G q_{n-1}(z)\overline{G_{n}(z)}\,dA(z)
=\frac{1}{n+1}\int_G q_{n-1}(z)\overline{F_{n+1}^\prime(z)}\,dA(z)\\
&=\frac{\pi}{n+1}\,\frac{1}{2\pi i}\int_\Gamma q_{n-1}(z)\overline{F_{n+1}(z)}\,dz.
\end{alignedat}
\end{equation*}
Hence, from (\ref{eq:PhinFn}),
$$
\frac{n+1}{\pi}\|q_{n-1}\|^2_{L^2(G)}
=\frac{1}{2\pi i}\int_\Gamma q_{n-1}(z)\overline{E_{n+1}(z)}\,dz
+\frac{1}{2\pi i}\int_\Gamma q_{n-1}(z)\overline{\Phi^{n+1}(z)}\,dz,
$$
and the result (\ref{eq:betanEn}) follows, because the last integral vanishes, as it can be readily seen after replacing $\overline{\Phi^{n+1}(z)}$ by $1/\Phi^{n+1}(z)$  and applying the residue theorem.

Next, we recall that $E_{n+1}$ is analytic in $\Omega$, including $\infty$, and continuous on $\overline{\Omega}$,
and that $H_n\in L_a^2(\Omega)\cap E^1(\Omega)$. The result (\ref{eq:epsnEn}) follows from the application of
Green's formula in the unbounded domain $\Omega$ and (\ref{eq:GnFn+1}). That is,
\begin{equation*}
\begin{alignedat}{1}
-\frac{1}{2\pi i}\,\int_\Gamma H_n(z)\overline{E_{n+1}(z)}\,dz
&=\frac{1}{\pi}\,\int_\Omega H_n(z)\overline{E^\prime_{n+1}(z)}\,dA(z)
=\frac{n+1}{\pi}\|H_n\|^2_{L^2(\Omega)}.
\end{alignedat}
\end{equation*}
\end{proof}

It turns out that the strong asymptotic error $\alpha_n$ has a very simple connection with the
quantities $\beta_n$ and $\varepsilon_{n}$, namely,
$$
\alpha_n=\beta_n+\varepsilon_{n}.
$$
(This, actually, explains the presence of the fractional term $(n+1)/\pi$ in the definition of $\beta_n$ and $\varepsilon_{n}$ above.)
\begin{lemma}\label{lem:alphabetaeps}
Assume that the boundary $\Gamma$ is rectifiable. Then, for any $n\in\mathbb{N}$, it holds that
\begin{equation}\label{eq:alphabetaeps}
\frac{n+1}{\pi}\frac{\gamma^{2(n+1)}}{\lambda_n^2}=1-\left(\beta_n+\varepsilon_{n}\right).
\end{equation}
\end{lemma}
\begin{proof}
Green's formula, in conjunction with (\ref{eq:GnFn+1}), yields:
\begin{equation*}
\|G_n\|_{L^2(G)}^2=\frac{1}{n+1}\,\int_G G_{n}(z)\overline{F^\prime_{n+1}(z)}\,dA(z)
=\frac{\pi}{n+1}\frac{1}{2\pi i}\,\int_\Gamma G_{n}(z)\overline{F_{n+1}(z)}\,dz.
\end{equation*}
Next, replace in the last integral $G_n(z)$ and $F_{n+1}(z)$ by their counterparts given in the splittings
(\ref{eq:PhinFn}) and (\ref{eq:Endef}), respectively, to obtain:
\begin{equation*}
\begin{alignedat}{1}
\int_\Gamma G_{n}(z)\overline{F_{n+1}(z)}\,dz
&=\int_\Gamma \Phi^n(z)\Phi^\prime(z)\overline{\Phi^{n+1}(z)}dz
+\int_\Gamma \Phi^n(z)\Phi^\prime(z)\overline{E_{n+1}(z)}dz\\
&+\int_\Gamma H_n(z)\overline{\Phi^{n+1}(z)}dz
+\int_\Gamma H_n(z)\overline{E_{n+1}(z)}dz.
\end{alignedat}
\end{equation*}
It therefore follows from Lemma~\ref{lem:usefull-1} that
\begin{equation*}
\|G_n\|_{L^2(G)}^2=\frac{\pi}{n+1}\left[1+\frac{1}{2\pi i}\,\int_\Gamma H_{n}(z)\overline{E_{n+1}(z)}\,dz\right],
\end{equation*}
which, in view of Lemma~\ref{lem-usefull-2}, yields the relation
\begin{equation}\label{eq:Gnepsn}
\|G_n\|_{L^2(G)}^2=\frac{\pi}{n+1}\,\left(1-\varepsilon_n\right),\quad n\in\mathbb{N}.
\end{equation}

At the other hand, since $ p_n\bot\,q_{n-1}$, we have from Pythagoras' theorem  that
\begin{eqnarray}\label{eq:pytha}
\|G_n\|^2_{L^2(G)}=\|\frac{\gamma^{n+1}}{\lambda_n}\,p_n+q_{n-1}\|^2_{L^2(G)}
       =\frac{\gamma^{2(n+1)}}{\lambda_n^2}+\|q_{n-1}\|^2_{L^2(G)},
\end{eqnarray}
and (\ref{eq:alphabetaeps}) follows by comparing (\ref{eq:Gnepsn}) with (\ref{eq:pytha}) and using the definition of
$\beta_n$ in (\ref{eq:betandef}).
\end{proof}

\begin{remark}\label{rem:beta-eps}
It follows immediately from (\ref{eq:alphabetaeps}) and (\ref{eq:Gnepsn}) that
\begin{equation}\label{eq:beta-eps}
0\le\beta_n+\varepsilon_n <1, \quad 0\le\varepsilon_n <1\quad and\quad 0\le\beta_n<1.
\end{equation}
In particular, these inequalities lead to the following three estimates
\begin{equation}\label{eq:|Gn|-rect}
\|G_n\|_{L^2(G)}\le\sqrt{\frac{\pi}{n+1}},\quad n\in\mathbb{N},
\end{equation}
and
\begin{equation}\label{eq:|qn|-rect}
\|q_{n-1}\|_{L^2(G)}<\sqrt{\frac{\pi}{n+1}},\quad
\|H_n\|_{L^2(\Omega)}<\sqrt{\frac{\pi}{n+1}},\quad n\in\mathbb{N},
\end{equation}
provided that $\Gamma$ is rectifiable.
The inequality in (\ref{eq:|Gn|-rect}) is sharp, as the case $G\equiv\mathbb{D}$ shows.
Furthermore, Lemma~\ref{lem-usefull-2} implies that $\beta_n$ and $\varepsilon_n$ vanish simultaneously
if $G$ is a disk.
\end{remark}

\subsection{Results for quasiconformal boundary}\label{sec2:qc}
For the next three results we need additional assumptions on $\Gamma$. Their respective proofs
are given in Section~\ref{proofs-qc}. The first of them is essential in the proof of Theorem~\ref{thm:betan} below,
and is of independent interest, in the sense it provides an estimate for the integral on $\Gamma$ of the product of two
functions (one of them defined on $\overline{G}$ and the other on $\overline{\Omega}$) in terms of associated
$L^2$-norms in $G$ and $\Omega$.

\begin{lemma}\label{lem:useful-1}
Assume that $\Gamma$ is quasiconformal and rectifiable. Then, for any $f$ analytic in $G$, continuous on
$\overline{G}$ and $g$ analytic in $\Omega$, continuous on $\overline{\Omega}$, with $g^\prime\in L^2(\Omega)$,
there holds that
\begin{equation}\label{eq:useful-1}
\left|\frac{1}{2i}\int_\Gamma f(z)\overline{g(z)}dz \right|\le
\frac{k}{\sqrt{1-k^2}}\,\|f\|_{L^2(G)}\|g^\prime\|_{L^2(\Omega)},
\end{equation}
where $k$ is the reflection factor of $\Gamma$ defined in (\ref{eq:refcoef}).
\end{lemma}
It is readily verified that in the case when $\Gamma$ is a circle then both sides of (\ref{eq:useful-1}) vanish.

The second result shows that the sequence $\{\beta_n\}$ is dominated by the sequence
$\{\varepsilon_n\}$. Note, in particular, that $\beta_n$, $\varepsilon_n$ and $k$ vanish simultaneously
if $\Gamma$ is a circle.
\begin{theorem}\label{thm:betan}
Assume that $\Gamma$ is quasiconformal and rectifiable. Then, for any $n\in\mathbb{N}$, it holds that
\begin{equation}\label{eq:betan}
0\le\beta_n\le \frac{k^2}{1-k^2}\,\,\varepsilon_{n},
\end{equation}
where $k$ is the reflection factor of $\Gamma$.
\end{theorem}

The third result relates the decay of $\{\varepsilon_n\}$ to that of the coefficients of the exterior conformal
map $\Psi$ and is essential in the proof of Theorem~\ref{thm:alphange}.
\begin{theorem}\label{lem:epsnge}
Assume that $\Gamma$ is quasiconformal. Then, for any $n\in\mathbb{N}$, it holds that
\begin{equation}\label{eq:epsge}
\varepsilon_{n}\ge\,\frac{\pi\,(1-k^2)}{A(G)}\,(n+1)\,|b_{n+1}|^2,
\end{equation}
where $A(G)$ denotes the area of $G$ and $k$ is the reflection factor of $\Gamma$.
\end{theorem}

\subsection{Results for piecewise analytic boundary}\label{sec2:pw-analytic}
The next two theorems are established for $\Gamma$ \textit{piecewise analytic without cusps}. This means that
$\Gamma$ consists of a finite number of analytic arcs, say $N$, that meet at corner points $z_j$, $j=1,\ldots,N$, where
they form exterior angles $\omega_j\pi$, with $0<\omega_j<2$. The proofs of these theorems are given in
Section~\ref{proofs-pa}.

The relation (\ref{eq:pnPhinPhip}) reveals that in order to derive the strong asymptotics for $p_n(z)$ in $\Omega$,
we need suitable estimates for $q_{n-1}(z)$ and $H_n(z)$ there. For $q_{n-1}(z)$ this is provided by
Corollary~\ref{lem:PolyLemma} below. Regarding $H_n(z)$, we can use the estimate (\ref{eq:Hn-decay-rect}), which is valid
for $\Gamma$ rectifiable.
However, under the current assumption on $\Gamma$ more can be obtained.
\begin{theorem}\label{thm:HnOmgae}
Assume that $\Gamma$ is piecewise analytic without cusps. Then, for any $n\in\mathbb{N}$, it holds that
\begin{equation}\label{eq:HnOmega}
|H_n(z)|\le\frac{c_2(\Gamma)}{\dist(z,\Gamma)}\,\frac{1}{n},\quad z\in\Omega,
\end{equation}
where $c_2(\Gamma)$ depends on $\Gamma$ only.
\end{theorem}
Regarding the $L^2$-norm of $H_n$ we have the following estimate; cf.\ (\ref{eq:|qn|-rect}).
\begin{theorem}\label{thm:epsn}
Assume that $\Gamma$ is piecewise analytic without cusps. Then, for any $n\in\mathbb{N}$, it holds that
\begin{equation}\label{eq:epsn}
\|H_n\|_{L^2(\Omega)}  \le c_3(\Gamma)\,\frac{1}{n},
\end{equation}
where $c_3(\Gamma)$ depends on $\Gamma$ only.
\end{theorem}
It is interesting to note an uniformity aspect in both the estimates (\ref{eq:HnOmega}) and (\ref{eq:epsn}),
in the sense that the geometry of
$\Gamma$, as it is measured by the values of $\omega_j\pi$, does not influence the way that
$H_n(z)$ and $\|H_n\|_{L^2(\Omega)}$
tend to zero. This is somewhat surprising, when compared with similar results in Approximation Theory for domains with
corners, and it can be attributed to the fact that the effect of $\omega_j$'s \lq\lq cancels out" in the
representation (\ref{eq:HnIntRep}) of $H_n(z)$, see  (\ref{eq:Iji}) and Remark~\ref{rem:deltaj} below.

We conclude this section with a simple consequence of Theorems~\ref{thm:betan} and \ref{thm:epsn}.
\begin{corollary}\label{cor:qn-1L2}
Assume that $\Gamma$ is piecewise analytic without cusps. Then, for any $n\in\mathbb{N}$, there holds that
\begin{equation}\label{eq:en-decay}
0\le\varepsilon_n\le c_4(\Gamma)\,\frac{1}{n},
\end{equation}
and
\begin{equation}\label{eq:qn-1L2}
\|q_{n-1}\|_{L^2(G)}\le c_5(\Gamma)\,\frac{1}{n},
\end{equation}
where $c_4(\Gamma)$ and  $c_5(\Gamma)$ depend on $\Gamma$ only.
\end{corollary}
A comparison between (\ref{eq:qn-1L2}) and (\ref{eq:|qn|-rect}) reveals the gain in the rate of decay of
$\|q_{n-1}\|_{L^2(G)}$ under  the additional assumption the $\Gamma$.

\section{A Polynomial lemma}\label{sec:poly-est}
\setcounter{equation}{0}
In the proof of Theorem~\ref{thm:finepn} we require an estimate for the growth of the polynomial $q_{n-1}(z)$ in
$\Omega$,
in terms of its $L^2$-norm in $G$. This is the purpose of the next lemma, which is of independent interest.
Its own proof is given in Section~\ref{sec:proof-lem:PolyLemma} below.
We use $\mathbb{P}_n$  to denote the space of the polynomials of degree up to $n$.
\begin{lemma}\label{lem:PolyLemma}
Assume that $\Gamma$ is quasiconformal and rectifiable. Then, for any $P\in\mathbb{P}_n$, it holds that
\begin{equation}\label{eq:PolyLemma}
|P(z)|\le\frac{1}{\dist(z,\Gamma)\sqrt{1-k^2}}\,\,\sqrt{\frac{n+1}{\pi}}\,
\|P\|_{L^2(G)}\,|\Phi(z)|^{n+1},\quad z\in\Omega,
\end{equation}
where $k$ is the reflection factor of $\Gamma$.
\end{lemma}
Regarding sharpness of the inequality (\ref{eq:PolyLemma}), we note that the order $1/2$ of $n$ cannot be improved in
general, as the the choice $P\equiv p_n$ and the strong asymptotics for smooth $\Gamma$ of Section~\ref{section:intro}
show.
Furthermore, the constant term is asymptotically optimal for $z\to\infty$, as the choice $P(z)=z^n$, with
$G=\mathbb{D}$ (hence $k=0$) shows.

Lemma~\ref{eq:PolyLemma} should be compared with the following well-known result, which gives the growth of a polynomial
in terms of its uniform norm on $\overline{G}$. Hereafter we use $\|\cdot\|_K$ to denote the uniform norm on the set $K$.
\begin{mylemma}[Bernstein-Walsh]\label{lem:B-W}
For any $P\in\mathbb{P}_n$, it holds that
\begin{equation}\label{eq:B-W}
|P(z)|\le\|P\|_{\overline{G}}\,\,|\Phi(z)|^{n},\quad z\in\Omega.
\end{equation}
\end{mylemma}
We note that the inequality (\ref{eq:B-W}) is valid under more general assumption for $\overline{G}$; see, e.g.,
\cite[p.~153]{ST}.
We also note the following norm-comparison result, which was quoted by Suetin in \cite[p.~38]{Su74}, under the assumption
that $\Gamma$ is smooth:
\begin{equation}\label{eq:Suetin-ineq}
\|P\|_{\overline{G}}\le c(\Gamma)\,n\,\|P\|_{L^2(G)},\quad P\in\mathbb{P}_n.
\end{equation}

To underline the importance of Lemma~\ref{lem:PolyLemma} for our work here, we observe that the combination of
(\ref{eq:B-W}) with (\ref{eq:Suetin-ineq}) gives the estimate
\begin{equation}
|P(z)|\le c(\Gamma)\,n\,\|P\|_{L^2(G)}|\Phi(z)|^{n},\quad z\in\Omega,
\end{equation}
and this for $P\equiv q_{n-1}$, together with Corollary~\ref{cor:qn-1L2}, yields
\begin{equation}\label{eq:unfor}
|q_{n-1}(z)|\le c(\Gamma)\,|\Phi(z)|^{n},\quad z\in\Omega,
\end{equation}
provided $\Gamma$ is smooth.
Unfortunately, (\ref{eq:unfor}) is not adequate for delivering
the strong asymptotics for $p_n(z)$, even for smooth $\Gamma$; see  the proof of
Theorem~\ref{thm:finepn} in Section~\ref{proofs-main}.

At the other hand, the combination of Lemma~\ref{eq:PolyLemma} with Corollary~\ref{cor:qn-1L2} yields
the following finer estimate, which suffices to convey that $A_n(z)=O(1/\sqrt{n})$
in (\ref{eqinthm:finepn}):
\begin{corollary}\label{cor:qn-in-Omega}
Assume that $\Gamma$ is piecewise analytic without cusps. Then, for any $n\in\mathbb{N}$, it holds that
\begin{equation}\label{eq:qn-in-Omega}
|q_{n-1}(z)|\le\frac{c_1(\Gamma)}{\dist(z,\Gamma)}\,\frac{1}{\sqrt{n}}\,|\Phi(z)|^{n},
\quad z\in\Omega.
\end{equation}
where $c_1(\Gamma)$ depends on $\Gamma$ only.
\end{corollary}

\section{Proofs for quasiconformal boundary}\label{proofs-qc}
\setcounter{equation}{0}
Assume now that $\Gamma$ is a quasiconformal curve. Our arguments in this section are based on the use of a $K$-\textit{quasiconformal reflection} $y:\overline{\mathbb{C}}\to\overline{\mathbb{C}}$ defined, for some $K\ge 1$, by $\Gamma$ and a fixed point $a$ in $G$.
Below, we collect together some well-known properties of $y(z)$ which are important for our work here and we refer
to the four monographs \cite{Ah66}, \cite{LV}, \cite{ABD} and \cite{AstalaIwaniecMartin}, for a concise account of
results in Quasiconformal Mapping Theory; see also \cite[\S 6]{Be77}.
\begin{remark}[Properties of quasiconformal reflection] \label{lem:prop-qc}
With the above notations it holds that:
\begin{enumerate}[\qquad]\itemsep=0pt
\item[(A1)]
$\overline{y}$ is a $K$-quasiconformal mapping $\overline{\mathbb{C}}\to\overline{\mathbb{C}}$;
\item[(A2)]
$y(G)=\Omega$, $y(\Omega)=G$, with $y(a)=\infty$ and $y(\infty)=a$;
\item[(A3)]
$y(z)=z$, for every $z\in\Gamma$ and $y(y(z))=z$, for all $z\in\mathbb{C}$.
\end{enumerate}
\end{remark}
For a function
$f:\overline{\mathbb{C}}\to\overline{\mathbb{C}}$ we use the notation $f_{z}$ and $f_{\overline{z}}$ to denote its formal complex derivatives
$$
f_z:=\frac{\partial f}{\partial z}=
\frac{1}{2}\left(\frac{\partial f}{\partial x}-i\frac{\partial f}{\partial x}\right)
\quad\text{and}\quad
f_{\overline{z}}:=\frac{\partial f}{\partial \overline{z}}=
\frac{1}{2}\left(\frac{\partial f}{\partial x}+i\frac{\partial f}{\partial x}\right).
$$
We recall that $f_z=f^\prime$ and $f_{\overline{z}}=0$, whenever $f$ is analytic, and
the two identities
\begin{equation}\label{eq:dbar-rule}
\overline{f_z}=\overline{f}_{\overline{z}}\quad\text{and}\quad
\overline{f_{\overline{z}}}=\overline{f}_z.
\end{equation}
We also recall the chain rule for formal derivatives that is, if $\zeta=g(z)$, then
\begin{equation}\label{eq:chain-rule-dbar}
(f\circ g)_{\overline{z}}=f_{\zeta}(g(z))\,g_{\overline{z}}(z)+f_{\overline{\zeta}}(g(z))\,\overline{g_z(z)}.
\end{equation}

The property (A1) implies that $y$ is a sense-reversing homeomorphism of $\overline{\mathbb{C}}$ onto $\overline{\mathbb{C}}$ satisfying, almost everywhere in $\mathbb{C}$,
\begin{equation}\label{eq:yzk}
\left|\frac{y_z}{y_{\overline{z}}}\right|=\left|\frac{\overline{y}_{\overline{z}}}{\overline{y}_{{z}}}\right|
\le k:=\frac{K-1}{K+1}<1.
\end{equation}
It further implies that, $y$ belongs to the Sobolev space $W^{1,2}_{loc}(\mathbb{C})$.
Recall that we refer to $k$ as the reflection factor of $\Gamma$ associated with $y$.

Let
$$
J(y(z)):=|y_z|^2-|y_{\overline{z}}|^2
$$
denote the Jacobian of the transformation $y:\overline{\mathbb{C}}\to\overline{\mathbb{C}}$, and note that $J(y(z))<0$,
because $y(z)$ is sense-reversing. It follows easily from (\ref{eq:yzk})
\begin{equation}\label{eq:yzJ}
|y_{\overline{z}}|^2\le\frac{-1}{1-k^2}\,J(y(z))\quad\textup{and}\quad
|y_z|^2\le\frac{-k^2}{1-k^2}\,J(y(z)),
\end{equation}
almost everywhere in $\mathbb{C}$. Thus, the change of variables
$\zeta=y(z)$ and the property (A2) yield immediately the two estimates
\begin{equation}\label{eq:intyzJ}
\int_\Omega|y_{\overline{z}}|^2dA(z)\le\frac{1}{1-k^2}A(G)\quad\textup{and}\quad
\int_\Omega|y_z|^2dA(z)\le\frac{k^2}{1-k^2}A(G),
\end{equation}
where $A(G)$ stands for the area of $G$.

\subsection{Proof of Lemma~\ref{lem:useful-1}}\label{sec:proof-lem:useful-1}
Since the function $g(z)$ is analytic in $\Omega$, it follows from (\ref{eq:dbar-rule}) and the chain rule (\ref{eq:chain-rule-dbar}) that
$$
\left[(\overline{g}\circ y)(z)\right]_{\overline{z}}=
\overline{g^\prime(y(z))}\,\overline{y_z},\quad z\in G.
$$
Hence, it is easy to verify that the function $\left[(\overline{g}\circ y)(z)\right]_{\overline{z}}$ is square
integrable in $G$.
This is a consequence of the assumption $g^\prime\in L^2(\Omega)$ and the
second inequality in (\ref{eq:yzJ}). Indeed, using the change of variables $\zeta=y(z)$ we have:
\begin{align}\label{eq:En-in-W}
\int_G\left|\left[(\overline{g}\circ y)(z)\right]_{\overline{z}}\right|^2dA(z)
=&\int_G\left|g^\prime(y(z))\right|^2\left|y_z\right|^2dA(z)\nonumber \\
\le& \frac{-k^2}{1-k^2}\,\int_G \left|{g^\prime}(y(z))\right|^2 J(y(z))\,dA(z)\nonumber\\
=&\frac{k^2}{1-k^2}\,\int_\Omega \left|{g^\prime}(\zeta)\right|^2 \,dA(\zeta).
\end{align}

Next, we set
$$
\eta_n:=\frac{1}{2i}\int_\Gamma f(z)\overline{g(z)}dz
$$
and observe that, since $y(z)=z$, for $z\in \Gamma$, $\eta_n$ can be written as
$$
\eta_n
=\frac{1}{2i}\,\int_\Gamma f(z)\,\overline{g(y(z))}\,dz.
$$

Finally, we note that the function $g(y(z))$ defines a
quasiconformal extension of $g(z)$ into $G$, which is continuous on $\overline{G}$.
Therefore, from the assumptions on $f$, $g$ and $\Gamma$ we conclude by means of
Green's formula that
\begin{equation*}
\begin{alignedat}{1}
\eta_n
=\int_G \left[f(z)\,(\overline{g}\circ y)(z)\right]_{\overline{z}}\,dA(z)
=\int_G f(z)\,\left[(\overline{g}\circ y)(z)\right]_{\overline{z}}\,dA(z),
\end{alignedat}
\end{equation*}
and the estimate (\ref{eq:useful-1}) then follows by applying the
Cauchy-Schwarz inequality to the last integral and using  (\ref{eq:En-in-W}).
\qed

\subsection{Proof of Theorem~\ref{thm:betan}}\label{sec:proof-thm:betan}
In view of Lemma~\ref{lem:Hn-in-L2} and (\ref{eq:GnFn+1}) we
note that $E_{n+1}^\prime\in L^2(\Omega)$ and apply the result of Lemma~\ref{lem:useful-1}
with  $f\equiv q_{n-1}$ and $g\equiv E_{n+1}$ to the expression of
$\beta_n$ given by (\ref{eq:betanEn}) to we obtain:
$$
\beta_n\le \frac{k}{\sqrt{1-k^2}}\,\frac{1}{\pi}\,\|q_{n-1}\|_{L^2(G)}\|E^\prime_{n+1}\|_{L^2(\Omega)}.
$$
Therefore, using the definition of $\beta_n$ and
$\varepsilon_n$ in (\ref{eq:betandef}) and (\ref{eq:epsndef}), we conclude that
\begin{equation*}
\begin{alignedat}{1}
\beta_n^2
&\le \frac{k^2}{1-k^2}\,\frac{(n+1)^2}{\pi^2}\,\|q_{n-1}\|^2_{L^2(G)}\|H_{n}\|^2_{L^2(\Omega)}\\
&=\frac{k^2}{1-k^2}\,\beta_n \,\varepsilon_n,
\end{alignedat}
\end{equation*}
which yields at once the required estimate (\ref{eq:betan}).
\qed

\subsection{Proof of Theorem~\ref{lem:epsnge}}\label{sec:proof-lem:epsnge}
Assume that $R>1$ is large enough so that the expansion (\ref{eq:Endef}) is valid for all $z\in L_R$.
Then, from the residue theorem and the splitting (\ref{eq:PhinFn}) we have
\begin{equation*}\label{eq:c_1-1}
c_1^{(n+1)}=\frac{1}{2\pi i}\,\int_{L_R}E_{n+1}(z)dz=-\frac{1}{2\pi i}\,\int_{L_R}\Phi^{n+1}(z)dz.
\end{equation*}

Next, by differentiating the expansion (\ref{eq:Psi}) of $\Psi(w)$ and applying again the residue theorem we see that
\begin{equation}\label{eq:c_1-2}
-(n+1)b_{n+1}=\frac{1}{2\pi i}\,\int_{|w|=R}w^{n+1}\Psi^\prime(w)dw=\frac{1}{2\pi i}\,\int_{L_R}\Phi^{n+1}(z)dz.
\end{equation}
Therefore, for any $n\in\mathbb{N}$,
\begin{equation*}
c_1^{(n+1)}=(n+1)b_{n+1}.
\end{equation*}
This, in view of (\ref{eq:Endef}) and (\ref{eq:GnFn+1}) shows that $a_2^{(n)}=-b_{n+1}$, where $a_2^{(n)}$ is the
coefficient of $1/z^2$ in the expansion (\ref{eq:Hndef}) of $H_n(z)$.
Hence, another application of the residue theorem yields that
\begin{equation*}
-b_{n+1}=\frac{1}{2\pi i}\,\int_{L_R}H_{n}(z)zdz=\frac{1}{2\pi i}\,\int_{\Gamma}H_{n}(z)zdz.
\end{equation*}

Furthermore, by using the fact that $y(z)=z$, for $z\in\Gamma$, and the  properties
of $H_n(z)$ and $y(z)$ in $\Omega$, we obtain  with the help of Green's formula in the unbounded
domain $\Omega$:
\begin{equation}\label{eq:bn+1Hn}
b_{n+1}=-\frac{1}{2\pi i}\,\int_{\Gamma}H_{n}(z)y(z)dz=
\frac{1}{\pi}\int_\Omega H_n(z) y_{\overline{z}}\,dA(z).
\end{equation}
The last integral can be estimated by means of the Cauchy-Schwarz inequality and the first inequality
in (\ref{eq:intyzJ}). Indeed,
\begin{equation*}
\begin{alignedat}{1}
\left|\int_\Omega H_n(z) y_{\overline{z}}\,dA(z)\right|
&\le\|H_n\|_{L_2(\Omega)} \left[\frac{1}{1-k^2}A(G)\right]^{1/2},
\end{alignedat}
\end{equation*}
and the required result emerges from (\ref{eq:bn+1Hn}) and the definition of $\varepsilon_n$.
\qed

\subsection{Proof of Lemma~\ref{lem:PolyLemma}}\label{sec:proof-lem:PolyLemma}
Let $P\in\mathbb{P}_n$ and fix $z\in\Omega$. Then, the function $P(z)/\Phi^{n+1}(z)$ is analytic in $\Omega$, continuous
on $\Gamma$ and vanishes at $\infty$. Hence, from Cauchy's formula and the property $y(\zeta)=\zeta$, for
$\zeta\in\Gamma$, we have
\begin{equation*}
\frac{P(z)}{\Phi^{n+1}(z)}
=-\frac{1}{2\pi i}\int_\Gamma\frac{g(\zeta)\,d\zeta}{\Phi^{n+1}(\zeta)}
=-\frac{1}{2\pi i}\int_\Gamma\frac{g(\zeta)\,d\zeta}{(\Phi^{n+1}\circ y)(\zeta)},
\end{equation*}
where $g(\zeta):=P(\zeta)/(\zeta-z)$. Now, the function $1/\Phi^{n+1}\circ y$ is continuous on $\overline{G}$, and its
$\partial/\partial_{\overline{z}}$ derivative belongs to $L^2(G)$; see (\ref{eq:Phip/Phin}) below.
Hence, from Green's formula we have that
\begin{align}\label{eq:P/Phin}
\frac{P(z)}{\Phi^{n+1}(z)} &= -\frac{1}{\pi}\,\int_G
\left[\frac{g(\zeta)}{(\Phi^{n+1}\circ y)(\zeta)}\right]_{\overline{\zeta}}dA(\zeta)\nonumber \\
&= \frac{n+1}{\pi}\,\int_G
g(\zeta)\,\frac{\Phi^\prime(y(\zeta))\,y_{\overline{\zeta}}}{(\Phi^{n+2}\circ y)(\zeta)}dA(\zeta),
\end{align}
where we made use of the fact that $g$ is analytic on $\overline{G}$. Next, using (\ref{eq:yzJ}) it is readily seen that
\begin{align}\label{eq:Phip/Phin}
\int_G\frac{|\Phi^\prime(y(\zeta))|^2\,|y_{\overline{\zeta}}|^2}{|(\Phi^{n+2}\circ y)(\zeta)|^2} dA(\zeta)
&\le \frac{-1}{1-k^2}\int_G\frac{|\Phi^\prime(y(\zeta))|^2\,
J(y(\zeta))}{|(\Phi^{n+2}\circ y)(\zeta)|^2}dA(\zeta)\nonumber \\
&=\frac{1}{1-k^2}\int_\Omega\frac{|\Phi^\prime(t)|^2\,dA(t)}{|\Phi^{n+2}(t)|^2}\nonumber \\
&=\frac{1}{1-k^2}\int_\Delta\frac{dA(w)}{|w^{n+2}|^2}
=\frac{1}{1-k^2}\frac{\pi}{(n+1)}.
\end{align}
Obviously,
\begin{equation*}
\int_G|g(\zeta)|^2dA(\zeta)
\le\frac{\|P\|_{L^2(G)}^2}{\left(\dist(z,\Gamma)\right)^2},
\end{equation*}
and the result (\ref{eq:PolyLemma}) follows from (\ref{eq:Phip/Phin}) and the application of the Cauchy-Schwarz
inequality to the integral in (\ref{eq:P/Phin}).
\qed

\section{Proofs for piecewise analytic boundary}\label{proofs-pa}
\setcounter{equation}{0}

We recall our assumption that $\Gamma$ consists of $N$ analytic arcs, which meet at corner points $z_j$, $j=1,\ldots,N$,
forming there exterior angles $\omega_j\pi$, with $0<\omega_j<2$.

The basic idea underlying the work in this section is simple. Extend, using Schwarz reflection, $\Phi$ across each arc of $\Gamma$ inside $G$, so that this extension is conformal in the exterior of a piecewise analytic Jordan curve $\Gamma^\prime$, which shares with $\Gamma$ the same corners $z_j$ and otherwise lies in $G$. $\Gamma^\prime$ can be chosen so that $\Phi$ is analytic on $\Gamma^\prime$, apart from $z_j$. Hence, the four representations
(\ref{eq:FnIntRep})--(\ref{eq:EnIntRep}) and (\ref{eq:GnIntRep})--(\ref{eq:HnIntRep}) remain valid if $\Gamma$ is deformed to $\Gamma^\prime$. Next, divide  $\Gamma^\prime$ into two parts: a part $l$ containing arcs emanating from the corners $z_j$, and a part $\tau$
constituting the complement $\Gamma^\prime\setminus l$, so that there exists a compact set $B:=B(\Gamma)$ of $G$  which contains $\tau$. When
$\zeta\in\tau$, $\Phi(\zeta)^n$ decays geometrically to zero, i.e., $|\Phi(\zeta)|^n=O(\rho^n)$, for some
$\rho:=\rho(\Gamma)<1,$ and therefore its own contribution is negligible, when compared with the contribution of
$\Phi(\zeta)^n$, for $\zeta\in l$.
To make things more precise, we assume (as we may) that $l$ is formed by linear segments, and we number these two
segment meeting at $z_j$ by $l_j^i$, $i=1,2$; see Figure~\ref{fig:Gammapr}.

\begin{figure}[h]
\begin{center}
\includegraphics*[scale=0.65]{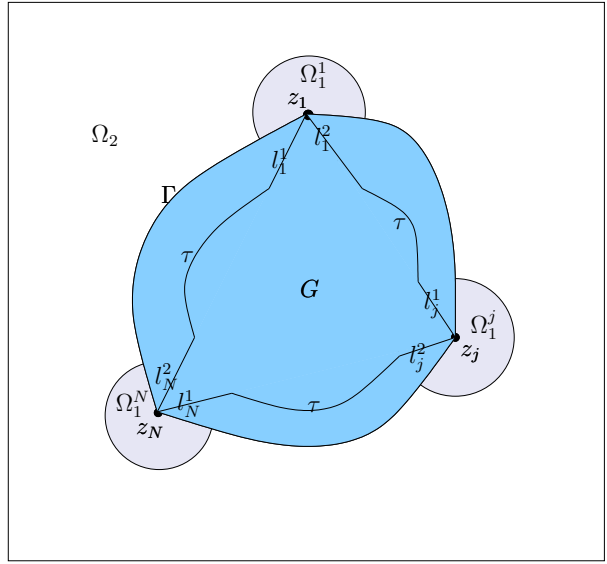}
\caption{The two decompositions: $\Gamma^\prime=l\cup\tau$ and $\Omega=\Omega_1\cup\Omega_2$.}
\label{fig:Gammapr}
\end{center}
\end{figure}

In the sequel, we make extensive use of the following four inequalities:
\begin{remark}[Behaviour of $\Phi$ near an analytic corner]\label{rem:Lehman}
For any $\zeta\in l_j^i$ there holds that:
\begin{enumerate}\itemsep=2pt
\item[\textup{(i)}]
$\displaystyle{|\Phi(\zeta)-\Phi(z_j)|\ge c\,|\zeta-z_j|^{1/\omega_j}}$;
\item[\textup{(ii)}]
$\displaystyle{|\Phi^\prime(\zeta)|\le c\,|\zeta-z_j|^{1/\omega_j-1}}$;
\item[\textup{(iii)}]
$\displaystyle{|\Phi(\zeta)|\le 1-c\,|\zeta-z_j|^{1/\omega_j}}$;
\item[\textup{(iv)}]
$\displaystyle{\dist(\zeta,\Gamma)\ge c\,|\zeta-z_j|}$.
\end{enumerate}
\end{remark}
(In Remark~\ref{rem:Lehman}, and below, we use the symbol $c$ generically in order to denote positive constants,
possibly different ones, that depend on $\Gamma$ only.)

The inequalities (i) and (ii) emerge from Lehman's asymptotic expansions for conformal mappings, near an analytic corner \cite{Lehman}.
The third inequality follows easily from (i), because reflection preserves angles. Finally, (iv) is a simple fact of conformal mapping geometry.

\subsection{Proof of Theorem~\ref{thm:HnOmgae}}\label{sec:proof-thm:HnOmgae}
The proof goes along similar lines as those taken in \cite{Ga01} for deriving an estimate for $F_n(z)$ in $G$, with one significant difference, though. Here $z$ lies in $\Omega$, rather than $G$, and thus $z$ is allowed to tend to $\Gamma$ without having to alter the curve $\Gamma^\prime$. As a consequence, the set $B$ defined above does not depend on $z$, and thus  $\dist(z,\tau)\ge\dist(z,B)>\dist(\Gamma,B)=c(\Gamma)$.

The details are as follows: From the discussion above, it is easy to see that, for $z\in\Omega$,
\begin{align}\label{eq:Hnparts}
H_n(z)&=\frac{1}{2\pi i}
\int_{\Gamma^\prime}\frac{\Phi^n(\zeta)\Phi^\prime(\zeta)}{\zeta-z}\,d\zeta\nonumber\\
  &=\frac{1}{2\pi i}\sum_j\int_{l_j^1\cup l_j^2}
    \frac{\Phi^n(\zeta)\Phi^\prime(\zeta)}{\zeta-z}\,d\zeta
     +\frac{1}{2\pi i}\int_{\tau}\frac{\Phi^n(\zeta)\Phi^\prime(\zeta)}{\zeta-z}\,d\zeta \nonumber\\
  &=\frac{1}{2\pi i}\sum_j\int_{l_j^1\cup l_j^2}\frac{\Phi^n(\zeta)\Phi^\prime(\zeta)}{\zeta-z}\,d\zeta
     +O(\rho^n),
\end{align}
for some $\rho:=\rho(\Gamma)<1$, independent of $z$. Hence, we only need to estimate the integral
$$
I_j^i:=\int_{l_j^i}\frac{\Phi^n(\zeta)\Phi^\prime(\zeta)}{\zeta-z}\,d\zeta.
$$

Let $s$ denote the arclength on ${l_j^i}$ measured from $z_j$. Then, Remark~\ref{rem:Lehman} yields the following two inequalities, which hold for any $\zeta\in l_j^i$:
\begin{equation}\label{eq:InePhi}
|\Phi(\zeta)|\le 1-cs^{1/\omega_j}<\exp(-cs^{1/\omega_j})\quad\textup{and}\quad
|\Phi^\prime(\zeta)|\le cs^{1/\omega_j-1}.
\end{equation}
Since $1/\omega_j>1/2$, these imply
\begin{equation}\label{eq:Iji}
|I_j^i|\le\frac{c}{\dist(z,\Gamma)}\int_0^\infty \textup{e}^{-cns^{1/\omega_j}}s^{1/\omega_j-1}\,ds=\frac{c\,\omega_j}{\dist(z,\Gamma)}\,\frac{1}{n},
\end{equation}
and the required estimate (\ref{eq:HnOmega}) follows from (\ref{eq:Hnparts}).
\qed

The next result is needed in establishing Theorem~\ref{thm:epsn}.
\begin{lemma}\label{lem:final}
With $\omega\in(0,2]$ and $k\in\mathbb{N}$, set $\delta:=k^{-\omega}$ and let
\begin{equation}\label{eq:Iome-del-k}
I(\omega,k):=\int_0^{\delta}\left[\int_r^\infty
{\textup{e}^{-ks^{1/\omega}}s^{1/\omega-2}}\,ds \right]^2rdr.
\end{equation}
Then,
\begin{equation}\label{eq:lemfinal}
I(\omega,k)\le\frac{c}{k^2}.
\end{equation}
\end{lemma}
(In the statement and proof of Lemma~\ref{lem:final} the positive constants $c$ depend on $\omega$ only.)
\begin{proof}
We consider separately the four complementary cases: (I) $\omega=1$, (II) $0<\omega<1$, (III) $\omega=2$ and (IV) $1<\omega<2$.

\medskip
\noindent
\textit{Case} (I): $\omega=1$. Note,
$$
I(1,k)=\int_0^{\delta}\left[\int_r^\infty\frac{\textup{e}^{-ks}}{s}\,ds \right]^2rdr
=\int_0^{\delta}E_1^2(kr)\,rdr,
$$
where $E_1(x)$, denotes the exponential integral $E_1(x):=\int_x^\infty{t^{-1}}{\textup{e}^{-t}}\,dt$,
with $x>0$. Using the formula $\int_0^\infty E_1^2(x)dx=2\log 2$ we thus have
$$
I(1,k)\le\delta\int_0^{\delta}E_1^2(kr)dr<\frac{\delta}{k}\int_0^\infty E_1^2(x)dx=\frac{c}{k^2}.
$$

\medskip
\noindent
\textit{Case} (II): $0<\omega<1$. Now $1/\omega>1$. Consequently, for $r>0$,
$$
\int_r^\infty\textup{e}^{-ks^{1/\omega}}s^{1/\omega-2}\,ds\le
\int_0^\infty\textup{e}^{-ks^{1/\omega}}s^{1/\omega-2}\,ds=\omega\,\Gamma(1-\omega)\,k^{\omega-1},
$$
where $\Gamma(x):=\int_0^\infty t^{x-1}\textup{e}^{-t}dt$ denotes the Gamma function with argument $x>0$.
This yields
\begin{equation}\label{eq:I-CaseII}
I(\omega,k)\le c\,\frac{\delta^2}{k^{2(1-\omega)}}=\frac{c}{k^2}.
\end{equation}

\medskip
\noindent
\textit{Case} (III): $\omega=2$. We note first the formula, valid for $r>0$,
$$
\int_r^\infty\textup{e}^{-ks^{1/2}}s^{-3/2}\,ds=
2\left({\textup{e}^{-k\,r^{1/2}}}{r^{-1/2}}-k\,E_1(k\,r^{1/2})\right).
$$
Therefore,
$$
I(2,k)< c\,\int_0^\infty\textup{e}^{-2k\,r^{1/2}}dr
+ c\,k^2\int_0^\infty E_1^2(k\,r^{1/2})\,rdr=\frac{c}{k^2}+k^2\frac{c}{k^4}=\frac{c}{k^2}.
$$

\medskip
\noindent
\textit{Case} (IV): $1<\omega<2$. The result for $1<\omega<2$ can be established as a special case of $\omega=1$ and $\omega=2$. To see this set $h(\omega,s):={\textup{e}^{-ks^{1/\omega}}s^{1/\omega-2}}$ and split the integral from $r$ to $\infty$ in (\ref{eq:Iome-del-k}) into three parts:
\begin{equation*}
\int_r^\infty h(\omega,s)\,ds
=\int_r^\delta h(\omega,s)\,ds+\int_\delta^1 h(\omega,s)\,ds+\int_1^\infty h(\omega,s)\,ds.
\end{equation*}
Next, observe that if $s\in(0,\delta)\cup(1,\infty)$, then $h(\omega,s)$ is an increasing function of $\omega$,
hence $h(\omega,s)\le h(2,s)$. At the other hand, when $s\in(\delta,1)$, then $h(\omega,s)$ is a decreasing function of $\omega$, thus $h(\omega,s)\le h(1,s)$.

Summing up, we therefore have
$$
\int_r^\infty h(\omega,s)\,ds
\le\int_r^\infty \textup{e}^{-ks^{1/2}}s^{-3/2}\,ds+\int_r^\infty \frac{\textup{e}^{-ks}}{s}\,ds,
$$
and the result (\ref{eq:lemfinal}) follows easily using the estimates given in Cases (I) and (III).
\end{proof}

\subsection{Proof of Theorem~\ref{thm:epsn}}\label{sec:proof-thm:epsn}
We choose positive quantities
\begin{equation}\label{eq:deltaj}
\delta_j=\delta_{n,j}:=c\,{n^{-\omega_j}},\quad j=1,\ldots,N,
\end{equation}
where $c$ is small enough so that any two of the $N$ domains $\Omega_1^{j}:=\{z\in\Omega:|z-z_j|<\delta_j\}$, are disjoint from each other. Next, we set $\Omega_{1}:=\cup_j^N\Omega_1^{j}$ and split $\Omega$ into two parts $\Omega_{1}$ and $\Omega_{2}$; see Figure~\ref{fig:Gammapr}.

Using this partition of $\Omega$, we express $\|H_n\|^2_{L^2(\Omega)}$ as the sum of two integrals over $\Omega_1$ and
$\Omega_2$. This gives
\begin{align}\label{eq:I1+I2}
\|H_n\|^2_{L^2(\Omega)}&=\int_{\Omega_{1}}|H_n(z)|^2dA(z)
+\frac{1}{(n+1)^2}\int_{\Omega_{2}}|E^\prime_{n+1}(z)|^2dA(z)\nonumber\\
&=:J_1(n)+J_2(n),
\end{align}
where we made use of (\ref{eq:GnFn+1}).
Hence, deriving the estimate (\ref{eq:epsn}) it now amounts to showing that: (a) $J_1(n)=O(1/n^2)$ and
(b) $J_2(n)=O(1/n^2)$.

\noindent
\textbf{(a)} Let
\begin{equation*}
T_j(n):=\int_{\Omega_{1}^j}|H_n(z)|^2dA(z),\quad j=1,\ldots,N,
\end{equation*}
so that,
\begin{equation}\label{eq:J1=sumint}
J_1(n)=\sum_{j=1}^NT_j(n).
\end{equation}

With $z\in\Omega_1^j$, set $r:=|z-z_j|$ and observe that, in view of Remark~\ref{rem:Lehman} (iv),
$|\zeta-z|\approx s+r$, if $\zeta\in l_j^1\cup l_j^2$, while $|\zeta-z|\ge c$, if $\zeta\in l_k^1\cup l_k^2$, with
$k\ne j$, where $s$ denote the arclength on ${l_j^i}$ measured from $z_j$. Consequently, since
$\Omega_1^j\subset\{z:|z-z_j|<\delta_j\}$, we obtain using (\ref{eq:Hnparts}) and (\ref{eq:InePhi}):
\begin{align}\label{eq:Tjn-le}
T_j(n)&\le
c\int_0^{\delta_j}\left[\int_0^\infty\frac{\textup{e}^{-cns^{1/\omega_j}}s^{1/\omega_j-1}}{r+s}ds
+\sum_{k\ne j}\int_0^\infty\textup{e}^{-cns^{1/\omega_k}}s^{1/\omega_k-1}ds\right]^2rdr \nonumber \\
&\le c\int_0^{\delta_j}
\left[\int_0^\infty\frac{\textup{e}^{-cns^{1/\omega_j}}s^{1/\omega_j-1}}{r+s}ds \right]^2rdr \nonumber\\
&+c\int_0^{\delta_j}\left[\sum_{k\ne j}\int_0^\infty\textup{e}^{-cns^{1/\omega_k}}s^{1/\omega_k-1}ds\right]^2rdr.
\end{align}
Since
$$
\int_0^\infty\textup{e}^{-cns^{1/\omega_k}}s^{1/\omega_k-1}ds=\frac{\omega_k}{c\, n},
$$
it follows from (\ref{eq:deltaj}) that
\begin{equation}\label{eq:Tjn-le-1}
\int_0^{\delta_j}\left[\sum_{k\ne j}\int_0^\infty\textup{e}^{-cns^{1/\omega_k}}s^{1/\omega_k-1}ds\right]^2rdr
\le\frac{c}{n^2}\int_0^{\delta_j}rdr\le\frac{c}{n^{2(1+\omega_j)}}.
\end{equation}
Next by splitting the integral on $(0,\infty)$ into the two parts $(0,r)$ and $(r,\infty)$ we get that
\begin{align*}
\int_0^{\delta_j}\left[\int_0^\infty\frac{\textup{e}^{-cns^{1/\omega_j}}s^{1/\omega_j-1}}{r+s}ds\right]^2rdr
&\le
c\int_0^{\delta_j}\left[\int_0^r\frac{\textup{e}^{-cns^{1/\omega_j}}s^{1/\omega_j-1}}{r}ds \right]^2rdr\\
&+c\int_0^{\delta_j}\left[\int_r^\infty{\textup{e}^{-cns^{1/\omega_j}}s^{1/\omega_j-2}}ds \right]^2rdr.
\end{align*}
Now we use the estimate
$$
\int_0^r\textup{e}^{-cns^{1/\omega_j}}s^{1/\omega_j-1}\,ds=\frac{c}{n}(1-\textup{e}^{-cnr^{1/\omega_j}})
<c\,r^{1/\omega_j},
$$
and the result of Lemma~\ref{lem:final} to deduce, respectively,
\begin{equation}\label{eq:last}
\int_0^{\delta_j}\left[\int_0^r\frac{\textup{e}^{-cns^{1/\omega_j}}s^{1/\omega_j-1}}{r}ds \right]^2rdr
<c\,\delta_j^{2/\omega_j}=c\,\frac{1}{n^2}
\end{equation}
and
$$
\int_0^{\delta_j}\left[\int_r^\infty{\textup{e}^{-cns^{1/\omega_j}}s^{1/\omega_j-2}}ds \right]^2rdr
\le c\,\frac{1}{n^2}.
$$

Summing up we conclude from (\ref{eq:Tjn-le}) that
$$
T_j(n)\le c\,\frac{1}{n^2},\quad\text{for }j=1,\ldots,N,
$$
which, in view of (\ref{eq:J1=sumint}),  leads to the required estimate
\begin{equation}\label{eq:I1le1/n}
J_1(n)\le\frac{c}{n^2}.
\end{equation}

\medskip
\noindent
\textbf{(b)}
By using Cauchy's integral formula for the derivative in (\ref{eq:EnIntRep}) and arguing as in
Section~\ref{sec:proof-thm:HnOmgae} we obtain, for $z\in\Omega$, that
\begin{eqnarray}\label{eq:En+1Om2}
E^\prime_{n+1}(z)&=&\frac{1}{2\pi i}
   \int_{\Gamma^\prime}\frac{\Phi^{n+1}(\zeta)}{(\zeta-z)^2}\,d\zeta\nonumber\\
   &=& \frac{1}{2\pi i}\sum_j\int_{l_j^1\cup l_j^2}\frac{\Phi^{n+1}(\zeta)}{(\zeta-z)^2}\,d\zeta
     +O(\rho^n),
\end{eqnarray}
with $\rho:=\rho(\Gamma)<1$ independent of $z$.

Assume now that $z\in\Omega_2$ and $\zeta\in l_j^1\cup l_j^2$, $j=1,\ldots,N$. Then, the triangle inequality and Remark~\ref{rem:Lehman} (iv) imply that $|\zeta-z|\ge c\,|z-z_j|$.
Thus, by using (\ref{eq:InePhi}) we get
\begin{eqnarray*}
\left|\int_{l_j^1\cup l_j^2}\frac{\Phi^{n+1}(\zeta)}{(\zeta-z)^2}\,d\zeta\right|
\le\frac{c}{|z-z_j|^2}\int_0^\infty \textup{e}^{-cns^{1/\omega_j}}ds
=\frac{c\,\Gamma(\omega_j)}{|z-z_j|^2}\,\frac{1}{n^{\omega_j}}.
\end{eqnarray*}
This, in conjunction with (\ref{eq:En+1Om2}), leads to the estimate
$$
\int_{\Omega_2}|E^\prime_{n+1}(z)|^2dA(z)
\le c\,\sum_j\frac{1}{n^{2\omega_j}}\int_{\Omega_2}\frac{dA(z)}{|z-z_j|^4}.
$$

Finally, since $\Omega_2\subset\{z:|z-z_j|\ge\delta_j\}$, we have from (\ref{eq:deltaj}) that
\begin{align}\label{eq:cannotdobetter}
\int_{\Omega_{2}}|E^\prime_{n+1}(z)|^2dA(z)
&\le c\sum_j\frac{1}{n^{2\omega_j}}\int_{|z-z_j|>\delta_j}\frac{dA(z)}{|z-z_j|^4}\nonumber\\
&=c\sum_j\frac{1}{n^{2\omega_j}\,\delta_j^2}=c,
\end{align}
and this, in view of the definition of $J_2(n)$ in (\ref{eq:I1+I2}), yields the required estimate
\begin{equation}\label{eq:I2le1/n}
J_2(n)\le\frac{c}{n^2}.
\end{equation}
\qed
\begin{remark}\label{rem:deltaj}
It is interesting to note that the choice for $\delta_j$ given by (\ref{eq:deltaj}) keeps the estimates
(\ref{eq:I1le1/n}) and (\ref{eq:I2le1/n}) in balance, in the sense that any other choice for $\delta_j$
will result to a weaker estimate for the decay of $\|H_n\|_{L^2(\Omega)}$, as a comparison of (\ref{eq:I-CaseII})
and (\ref{eq:last}) with (\ref{eq:cannotdobetter}) shows.
\end{remark}

\section{Proof of the main theorems}\label{proofs-main}
\setcounter{equation}{0}

\begin{proof}[Proof of Theorem~\ref{thm:finelambdan}.]
The assumption of the theorem implies that $\Gamma$ is quasiconformal and rectifiable. Hence, by
comparing (\ref{eqinthm:finelambdan}) with (\ref{eq:alphabetaeps}) we see that
\begin{equation}\label{eq:albeep}
\alpha_n=\beta_n+\varepsilon_n
\end{equation}
and the result (\ref{eqinthm:finelambdanii}) emerges immediately in view of  Theorem~\ref{thm:betan}
and Corollary~\ref{cor:qn-1L2}.
\end{proof}

\begin{proof}[Proof of Theorem~\ref{thm:finepn}.]
Theorem~\ref{thm:finelambdan} implies that
\begin{equation}\label{eq:lam/gam}
\frac{\lambda_n}{\gamma^{n+1}}=\sqrt{\frac{n+1}{\pi}}\left\{1+\xi_n\right\},\quad n\in\mathbb{N},
\end{equation}
where
\begin{equation}\label{eq:lam/gam-2}
0\le\xi_n\le c_1(\Gamma)\,\frac{1}{n}.
\end{equation}
Therefore, from (\ref{eq:pnPhinPhip}) we have, for $z\in\Omega$, that
\begin{equation*}
p_n(z)=\sqrt{\frac{n+1}{\pi}}\,\Phi^n(z)\Phi^\prime(z)
\left\{1+\xi_n\right\}\left\{1+\frac{H_n(z)}{\Phi^n(z)\Phi^\prime(z)}
-\frac{q_{n-1}(z)}{\Phi^n(z)\Phi^\prime(z)}\right\},
\end{equation*}
which, in comparison with (\ref{eqinthm:finepn}), gives the following explicit expression for the error
$A_n(z)$:
\begin{equation}\label{eq:A_-esti}
A_n(z)=\xi_n+\left\{1+\xi_n\right\}\frac{1}{\Phi^\prime(z)}\left\{\frac{H_n(z)}{\Phi^n(z)}
-\frac{q_{n-1}(z)}{\Phi^n(z)}\right\}.
\end{equation}
The required result (\ref{eqinthm:finepnii1}) then emerges by using the  estimates
(\ref{eq:lam/gam-2}), (\ref{eq:HnOmega})  and (\ref{eq:qn-in-Omega}), for $\xi_n$,
$H_n(z)$ and $q_{n-1}(z)$, respectively.
\end{proof}

\begin{proof}[Proof of Theorem~\ref{thm:alphange}.]
Immediately from (\ref{eq:albeep}) and Theorem~\ref{lem:epsnge}, since $\beta_n\ge 0$.
\end{proof}

\section{Applications}\label{sec:appl}
\setcounter{equation}{0}
Strong asymptotics for orthogonal polynomials with respect to measures supported on the real line have played a crucial
role in the development of the theory of orthogonal polynomials in $\mathbb{R}$.
In order to argue that this would be the case for Bergman polynomials as well, we present in briefly a
number of applications based on the strong asymptotics of Section~\ref{section:intro}
and the associated theory developed in Sections~\ref{sec:faber}--\ref{proofs-pa}.

\subsection{Zeros of the Bergman polynomials}\label{subsec:zeros}
A well-known result of Fejer asserts that \textit{all the zeros of} $\{p_n(z)\}_{n\in\mathbb{N}}$, \textit{are contained on the convex hull} $\textup{Co}(\overline{G})$ of $\overline{G}$. This was refined by Saff~\cite{Sa90} to the interior of $\textup{Co}(\overline{G})$. To these it should be added a result of Widom \cite{Wi67} to the effect that, \textit{on any closed subset $B$ of $\Omega\cap\textup{Co}(\overline{G})$ and for any $n\in\mathbb{N}$, the number of zeros of $p_n(z)$ on $B$ is bounded independently of $n$}. This of course, doesn't preclude the possibility that, if $B\neq\emptyset$, then $p_n(z)$ has a zero on $B$, for every $n\in\mathbb{N}$. The next theorem, which is a simple consequence of Theorem~\ref{thm:finepn}, shows that, under an additional assumption on $\Gamma$, the zeros of the sequence $\{p_n(z)\}_{n\in\mathbb{N}}$ cannot be accumulated in $\Omega$.
\begin{theorem}\label{thn:zeros}
Assume that $\Gamma$ is piecewise analytic without cusps. Then, for any closed set $B\subset\Omega$, there exists $n_0\in\mathbb{N}$, such that for $n\ge n_0$,
$p_n(z)$ has no zeros on $B$.
\end{theorem}

\subsection{Weak asymptotics}\label{subsec:weak}
The important class \textbf{Reg} of measures of orthogonality was introduced by Stahl and Totik in
\cite[Def.\ 3.1.2]{StTobo}. Since the area measure $dA$ on $G$ belongs to \textbf{Reg}, it follows that
\begin{equation}\label{eq:St-To:weak}
\lim_{n \to\infty}|p_n(z)|^{1/n}=|\Phi(z)|,
\end{equation}
locally uniformly in $\overline{\mathbb{C}}\setminus\textup{Co}(\overline{G})$; see \cite[Thm~3.1.1(ii)]{StTobo}.
The next theorem shows how this result  can be made more precise, under an additional assumption on the boundary.
\begin{theorem}
Assume that $\Gamma$ is piecewise analytic without cusps. Then,
$$
\lim_{n \to\infty}|p_n(z)|^{1/n}=|\Phi(z)|,
$$
locally uniformly in $\Omega$.
\end{theorem}
\begin{proof}
At once, after utilizing Theorem~\ref{thn:zeros} into \cite[Thm III.4.7]{ST}.
\end{proof}
For an account on weak asymptotics for Bergman polynomials defined by a system of disjoint Jordan curves we refer to \cite[Prop.~3.1]{GPSS}.

\subsection{Ratio asymptotics}\label{sec:ratio}
The following two corollaries are simple consequences of Theorems~\ref{thm:finelambdan} and \ref{thm:finepn}.
\begin{corollary}\label{cor:ratioln}
Assume that $\Gamma$ is piecewise analytic without cusps. Then, for any $n\in\mathbb{N}$, it holds that
\begin{equation}\label{eq:ratioln1}
\sqrt{\frac{n+1}{n+2}}\frac{\lambda_{n+1}}{\lambda_n}=\gamma+\varsigma_n,
\end{equation}
where
\begin{equation}\label{eq:ratioln2}
|\varsigma_n|\le c_1(\Gamma)\,\frac{1}{n}.
\end{equation}
\end{corollary}

\begin{corollary}\label{cor:ratiopn}
Under the assumptions of Corollary \ref{cor:ratioln}, for any $z\in\Omega$ and sufficiently large $n\in\mathbb{N}$,
it holds that
\begin{equation}\label{eq:ratiopn1}
\sqrt{\frac{n+1}{n+2}}\frac{p_{n+1}(z)}{p_n(z)}=\Phi(z)\left\{1+B_n(z)\right\},
\end{equation}
where
\begin{equation}\label{eq:ratiopn2}
|B_n(z)|\le \frac{c_2(\Gamma)}{\dist(z,\Gamma)|\Phi^\prime(z)|}\,\frac{1}{\sqrt{n}}
+c_3(\Gamma)\,\frac{1}{n}.
\end{equation}
\end{corollary}
\begin{remark}\label{rem:ratio}
The ratio asymptotics above are derived as consequence of Theorems~\ref{thm:finelambdan} and \ref{thm:finepn}.
Thus, we are obliged to assume that $\Gamma$ is piecewise analytic without cusps. Based, however, on substantial
numerical evidence (an instance is shown in Table~\ref{tab:cap-hypo} below) we believe that the ratio asymptotics hold,
in the sense that $\varsigma_n=o(1)$ and $B_n(z)=o(1)$, under weaker assumptions on $\Gamma$.
\end{remark}

\subsection{Stability of the Arnoldi GS for polynomials}\label{subsec:ArnoldiGS}
Let $\mu$ be a (non-trivial) finite Borel measure  supported on a compact (and infinite) subset $K$
of the complex plane, and let $\{p_n(z,\mu)\}_{n=0}^\infty$ denote the associated sequence of orthonormal polynomials
$$
p_n(z,\mu):=\lambda_n(\mu)z^n+\cdots,\quad \lambda_n(\mu)>0, \quad n=0,1,2,\ldots,
$$
generated by the inner product
$$
\langle f,g\rangle_\mu:=\int f(z)\overline{g(z)}d\mu(z).
$$

A standard way to construct the sequence $\{p_n(z,\mu)\}_{n=0}^\infty$, even to prove its existence and uniqueness,
is by using the Gram-Schmidt (GS) process. This process is designed to turn, in iterative fashion, any polynomial
sequence $\{P_n\}_{n=0}^\infty$ into an orthonormal sequence. The main ingredients in the computation are the complex
moments $\langle z^m,z^k\rangle_\mu$. The conventional to way apply the GS process is by choosing the monomials as the
starting up sequence, that is by setting $P_n(z)=z^n$.
Indeed, this was suggested (see, e.g., \cite[\S18.3--18.4]{He-V3}) and was eventually used (see, e.g., \cite{PW86})
by people working in Numerical Conformal Mapping, where the need for constructing orthonormal polynomials arises
from the application of the Bergman kernel method and its variants.

By the \textit{Arnoldi} GS we mean the application of the GS process in the following way:
At the $k$-step, where the orthonormal polynomial $p_k$ is to be constructed, use the polynomials
$\{p_0,p_1,\ldots,p_{k-1},zp_{k-1}\}$, rather than the monomials, as the starting up sequence.

Regarding the stability properties of the Arnoldi GS, we note that it is not difficult to show that
\begin{equation}\label{eq:insta-esti}
1\le I_n\le \|z\|_K\frac{\lambda_{n-1}^2(\mu)}{\lambda_n^2(\mu)},
\end{equation}
for the instability indicator
\begin{equation}\label{eq:In-def}
I_n:=\frac{\|P_n\|^2_{L^2{(G)}}}
{{\min_{P\in\text{span}(S_{n-1})}}\|P_n-P\|^2_{L^2{(G)}}}, \quad n\in\mathbb{N}
\end{equation}
introduced by Taylor in \cite{Ta78} for the purpose of measuring the
instability of the application of the GS process in orthonormalizing the set of polynomials
$S_n:=\{P_0,P_1,\ldots,P_n\}$.
Note that $I_n=1$, if $S_n$ is already an orthonormal set, while $I_n=\infty$, if $S_n$ is linearly depended.

In view of Corollary~\ref{cor:ratioln}, the estimate (\ref{eq:insta-esti}) implies that the Arnoldi GS process for
computing the Bergman polynomials of $G$ is stable, in the sense that the instability indicator $I_n$ does not increase
(in fact remains uniformly bounded)
with $n$. This is in sharp contrast with the
conventional GS, where $I_n$ \textit{increases geometrically fast} with $n$.
More specifically, the following estimate for the conventional GS was derived in \cite[Thm 3.1]{PW86}:
\begin{equation}\label{eq:insta-esti-PW}
c_4(\Gamma)L^{2n}\le I_n\le c_5(\Gamma)L^{2n},
\end{equation}
where $L:=\|z\|_\Gamma/\capGm$. Note that $L>1$, unless $G$ is a disk centered at the origin, where $L=1$.
For a comprehensive account on the damaging effects of the conventional GS process to the computation of
Bergman polynomials we refer to \cite{PW86}.

It is interesting to note that although Arnoldi's original paper \cite{Ar51} appeared in 1951, and the Arnoldi
implementation of the GS process was used in Numerical Linear Algebra since then, we first encountered
its implementation in connection with the computation of orthogonal polynomials much latter in \cite{GrRe-87},
where it was proposed for the computation of Szeg\H{o} polynomials without reference, however,
to its stability properties.

\subsection{Computation of $\Phi(z)$  and $\text{cap}(\Gamma)$}\label{sec:comp-cap}
Since $\textup{cap}(\Gamma)=b=1/\gamma$, Corollary~\ref{cor:ratioln}  provides the means for computing approximations to the capacity of $\Gamma$ by using only the leading coefficients of the Bergman polynomials.
Similarly, Corollary~\ref{cor:ratiopn} suggests a simple numerical method for computing approximations to the conformal
map $\Phi(z)$. This is quite appealing, in the sense that the Bergman polynomials, alone, suffice to provide approximations
to both interior conformal map $G\to\mathbb{D}$ (via the well-known Bergman kernel method) and exterior conformal
map $\Omega\to\Delta$, associated with the same Jordan curve. We refer to \cite{LySt} for the current state of the
convergence theory of the Bergman kernel method.
Regarding the exterior map we propose here the following approximation algorithm.

\medskip
\noindent {\tt Approximation of Capacities and Exterior Conformal Maps}
\begin{enumerate}[1.]
\itemsep=5pt
\item
Compute the complex moments
\begin{equation}\label{eq:mom}
\mu_{m,k}:=\langle z^m,z^k\rangle_G=\int_G z^m\overline{z}^k dA(z),\quad m,k=0,1,\ldots,n.
\end{equation}

\item
Employ the Arnoldi GS process
to construct the Bergman polynomials $\{p_k\}_{k=0}^n$ using the moments $\mu_{mk}$.
\item
Set
\begin{equation}\label{eq:approx-cap-Phi}
b^{(n)}:=\sqrt{\frac{n+1}{n}}\frac{\lambda_{n-1}}{\lambda_n}\quad
\text{and}\quad\Phi_n(z):=\sqrt{\frac{n}{n+1}}\frac{p_n(z)}{p_{n-1}(z)}.
\end{equation}
\item
Approximate $\textup{cap}(\Gamma)$ by $b^{(n)}$ and $\Phi(z)$ by $\Phi_n(z)$.
\end{enumerate}

\medskip

We demonstrate the performance of the above algorithm in the computation of capacities only.
We do so by presenting numerical results for two examples: (a) the canonical square with boundary $\Pi_4$, discussed in
Section~\ref{subsec:coef-esti} below, and (b) the 3-cusped hypocycloid with boundary $H_3$, defined by
(\ref{eq:Psi-hypom}) with $m=3$.
We note that $H_3$ does not satisfy the requirements of Corollary~\ref{cor:ratioln}.
The capacity of $\Pi_4$ is given explicitly in (\ref{eq:cap-square}),
while clearly, $\text{cap}(H_3)=1$. In both cases the complex moments are known explicitly.
The details of the presentation are as follows:

Let $t_n$ denote the error in approximating the capacity, i.e.,
\begin{equation}\label{eq:tn-dfn}
t_n:=b^{(n)}-\capGm.
\end{equation}
Since $\capGm=b$, it follows from Corollary~\ref{cor:ratioln} that
\begin{equation}\label{eq:tn-decay}
|t_n|\le c(\Gamma)\frac{1}{n},\quad n\in\mathbb{N}.
\end{equation}
In Tables~\ref{tab:cap-square}--\ref{tab:cap-hypo} we report the computed values of $b^{(n)}$ and $t_n$,
with $n$ varying from $100$ to $400$.
We also report the values of the parameter $s$, which is designed to test the hypothesis that
$|t_n|\approx 1/n^s$. All computations presented in this paper were carried out on a desktop PC, using the computing
environment MAPLE in high precision. Thus, in view of the stability properties of the Arnoldi GS process discussed in
Section~\ref{subsec:ArnoldiGS}, we expect all the figures quoted in the tables to be correct.

The numbers listed on the tables show that the proposed algorithm constitutes a valid method for computing
capacities. Is is interesting to note that in both cases the presented values of $b^{(n)}$ decay monotonically
to the capacity.
Also, the values of the parameter $s$ indicate clearly that for the case of the square $|t_n|\approx 1/n^2$. This
behaviour  can
be explained if $\alpha_n\approx 1/n$, for the strong asymptotic error of the leading coefficient.
For the case of the cusped hypocycloid however, no safe conclusions can be drawn for the behaviour of $t_n$ from the
reported values on Table~\ref{tab:cap-hypo}.

\begin{table}
\begin{tabular}{|c|c|c|c|}
\hline $ $
$n$& $b^{(n)}$ & $t_n$ & $s$  ${\vphantom{\sum^{\sum^{\sum^N}}}}$ \\*[3pt] \hline
100 & 0.834\,640\,612 & 1.37e-05  & -      \\
110 & 0.834\,638\,233 & 1.14e-05  & 1.9902     \\
120 & 0.834\,636\,420 & 9.58e-06  & 1.9911      \\
130 & 0.834\,635\,009 & 8.16e-06  & 1.9918      \\
140 & 0.834\,633\,888 & 7.04e-06  & 1.9924      \\
150 & 0.834\,632\,982 & 6.14e-06  & 1.9930      \\
160 & 0.834\,632\,341 & 5.39e-06  & 1.9934      \\
170 & 0.834\,631\,626 & 4.78e-06  & 1.9938      \\
180 & 0.834\,631\,111 & 4.26e-06  & 1.9942     \\
190 & 0.834\,630\,674 & 3.83e-06  & 1.9945      \\
200 & 0.834\,630\,301 & 3.46e-06  & 1.9949      \\
 \hline
\end{tabular}

\medskip
\caption{Square: The approximation $b^{(n)}$ of $\text{cap}(\Pi_4)=0.834\,626\,841\cdots$.}
\label{tab:cap-square}
\end{table}

\begin{table}
\begin{tabular}{|c|c|c|c|}
\hline $ $
$n$& $b^{(n)}$ & $t_n$ & $s$  ${\vphantom{\sum^{\sum^{\sum^N}}}}$ \\*[3pt] \hline
300 & 1.000\,117\,809 & 1.17e-04  & -      \\
310 & 1.000\,112\,347 & 1.12e-04  & 1.447   \\
320 & 1.000\,107\,296 & 1.07e-04  & 1.448  \\
330 & 1.000\,102\,615 & 1.02e-04  & 1.449  \\
340 & 1.000\,098\,267 & 9.82e-04  & 1.449  \\
350 & 1.000\,094\,219 & 9.42e-05  & 1.450  \\
360 & 1.000\,090\,443 & 9.04e-05  & 1.451  \\
370 & 1.000\,086\,914 & 8.69e-05  & 1.452  \\
380 & 1.000\,083\,610 & 8.36e-05  & 1.453  \\
390 & 1.000\,080\,511 & 8.05e-05  & 1.454  \\
400 & 1.000\,077\,600 & 7.76e-05  & 1.455  \\
\hline
\end{tabular}

\medskip
\caption{Hypocycloid: The approximation $b^{(n)}$ of $\text{cap}(\Pi_4)=1$.}
\label{tab:cap-hypo}
\end{table}

Based on  the important applications of the ratio asymptotics outlined above (see also Section~\ref{sec:ftrr})
we reckon that the solution of the following
problem will be of significance in developing further the theory of orthogonal polynomials in the complex plane.
\begin{problem}
Characterize all the measures of orthogonality $\mu$, with $\text{supp}(\mu)=K$, for which it holds:
\begin{equation}\label{eq:Ratdef}
\lim_{n\to\infty}\frac{\lambda_{n+1}(\mu)}{\lambda_n(\mu)}=\frac{1}{\textup{cap}(K)}.
\end{equation}
\end{problem}
Since the property $\mu\in\,$\textbf{Reg} is equivalent to
\begin{equation}\label{eq:Regdef}
\lim_{n\to\infty}\lambda_n^{1/n}(\mu)=\frac{1}{\textup{cap}(K)};
\end{equation}
see \cite[Thm 3.1.1]{StTobo}, it follows that the measures satisfying (\ref{eq:Ratdef}) form a subclass of
\textbf{Reg}.
We note, however, that there are known instances where the limit points of the sequence
$\{\lambda_{n+1}(\mu)/\lambda_n(\mu)\}_{n\in\mathbb{N}}$ constitute a finite set, as in the case of Bergman polynomials defined on a
system of disjoint symmetric lemniscates (see \cite[\S 7]{GPSS}), or where they fill up a whole interval, as in the case
of Szeg\H{o} polynomials defined on a system of disjoint smooth Jordan curves (see \cite[Thm 9.2]{Wi69}).

\subsection{Finite recurrence relations and Dirichlet problems}\label{sec:ftrr}
\begin{definition}\label{def:ftrr}
We say that the  polynomials $\{p_n\}_{n=0}^\infty$ satisfy an $(M+1)$-\textit{term recurrence relation}, if for any $n\geq M-1$,
$$
zp_n(z)=a_{n+1,n}p_{n+1}(z) + a_{n, n} p_n(z) + \cdots +
a_{n-M+1, n} p_{n-M+1}(z).
$$
\end{definition}
A direct application of the ratio asymptotics for $\{p_n\}_{n\in\mathbb{N}}$, given by Corollary~\ref{cor:ratiopn}, leads to the next two theorems. These refine, respectively, Theorems~2.2 and 2.1 of \cite{KhSt}, in the sense that they weaken the $C^2$-smoothness assumption on $\Gamma$. For their proof, it is sufficient to note that: (a) the two theorems are equivalent to each other and (b) the reason for assuming that $\Gamma$ is  $C^2$-smooth in Theorem~2.2 of \cite{KhSt} was to ensure the ratio asymptotics of the Bergman polynomials; see \cite[\S 4 Rem.~(i)]{KhSt}.
\begin{theorem}\label{thm:ftrr}
Assume that $\Gamma$ is piecewise analytic without cusps. If the Bergman polynomials $\{p_n\}_{n=0}^\infty$ satisfy an $(M+1)$-term recurrence relation, with some $M\ge 2$, then $M=2$ and $\Gamma$ is an ellipse.
\end{theorem}

\begin{theorem}\label{thm:DP}
Let $G$ be a bounded simply-connected domain with Jordan boundary $\Gamma$, which is piecewise analytic without cusps. Assume that there exists a positive integer $M:=M(G)$ with the property that the Dirichlet problem
\begin{equation}\label{eq:DP}
\left\{
\begin{alignedat}{2}
&\Delta u=0\quad &&\text{in} \ \ G, \\
&u=\overline{z}^mz^n\quad&&\text{on}\ \ \Gamma,
\end{alignedat}
\right.
\end{equation}
has a polynomial solution of degree $\le m(M-1)+n$ in $z$ and of degree $\le n(M-1)+m$ in $\overline{z}$, for all positive integers $m$ and $n$. Then $\Gamma$ is an ellipse and $M=2$.
\end{theorem}
Theorem~\ref{thm:DP} confirms a special case of the so-called Khavinson and Shapiro conjecture; see \cite{KL} for results reporting on the recent progress in this direction. We note that the equivalence between
the two properties \lq\lq the Bergman polynomials of $G$ satisfy a finite-term recurrence relation" and \lq\lq any Dirichlet problem in $G$, with polynomial data, possesses a polynomial solution" was first established in \cite{PuSt}.

\subsection{Shape recovery from partial measurements}\label{subsec:shape-rec}
Given a finite $n+1\times n+1$ section
\begin{equation}\label{eq:trun-mom}
[\mu_{m,k}]_{m,k=0}^n,\quad\mu_{m,k}:=\int_G z^m\overline{z}^k
dA(z),
\end{equation}
of the infinite complex moment matrix $[\mu_{m,k}]_{m,k=0}^\infty$,
associated with a bounded Jordan domain $G$, the \textit{Truncated Moments Problem} consists of computing an
approximation $\Gamma_n$ to its boundary $\Gamma$, by using only the data provided by (\ref{eq:trun-mom}).
Regarding existence and uniqueness, we note a result of Davis and Pollak \cite{DaPo} stating that the infinite matrix
$[\mu_{m,k}]_{m,k=0}^\infty$ defines uniquely the curve $\Gamma$.
Corollary~\ref{cor:ratiopn} and the discussion in Section~\ref{subsec:ArnoldiGS}, regarding the stability
of the Arnoldi GS process, suggest the following algorithm:

\medskip
\noindent {\tt Reconstruction from Moments Algorithm}
\begin{enumerate}[1.]
\itemsep=5pt
\item
Use the Arnoldi GS process to construct the Bergman polynomials $\{p_k\}_{k=0}^n$
from the given complex moments $\mu_{mk}$, $m,k=0,1,\ldots,n$.
\item
Compute the coefficients of the Laurent series expansion of the ratio
\begin{equation}\label{LauSerP}
\Phi_n(z):=\sqrt{\frac{n}{n+1}}\frac{p_n(z)}{p_{n-1}(z)}=\gamma^{(n)}z+\gamma_0^{(n)}+\frac{\gamma_1^{(n)}}{z}+
\frac{\gamma_2^{(n)}}{z^2}
+\cdots.
\end{equation}
\item
Revert the series (\ref{LauSerP}) using the explicit method described in \cite[p. 764]{Fa-Ol99}.
This leads to:
$$
b^{(n)}:=1/\gamma^{(n)}=\sqrt{\frac{n+1}{n}}\frac{\lambda_{n-1}}{\lambda_n},\quad
b_0^{(n)}:=-b^{(n)}\gamma_0^{(n)}/\gamma^{(n)}
$$
and
$$
\Psi_n(w):=b^{(n)}w+b_0^{(n)}+\frac{b_1^{(n)}}{w}+
\frac{b_2^{(n)}}{w^2}+\frac{b_3^{(n)}}{w^3}+\cdots+\frac{b_n^{(n)}}{w^n},
$$
where $-k\, b_k^{(n)}/b^{(n)}$, $k=1,2,\ldots,n$, is the coefficient of $1/z$ in the Laurent series expansion of $\left[\Phi_n(z)/\gamma^{(n)}\right]^k$ about infinity.
\item
Approximate $\Gamma$ by
$\Gamma_n:=\{z:\,z=\Psi_n(e^{it}),\,t\in[0,2\pi]\,\}$.
\end{enumerate}

\medskip

For applications to the 2D image reconstruction arising from tomographic data we refer to \cite{GPSS}.
Here we highlight the performance of the reconstruction algorithm by applying it to the recovery of three shapes,
where the defining  curves come from different classes: one analytic, one with corners and
one with cusps, for which the theory of Section~\ref{sec:ratio} does not apply.
In each case we start by computing a finite set of complex moments and then follow the four steps of the algorithm.
We note that in all three examples the complex moments are known explicitly.

In Figures~\ref{fig:ellipse}--\ref{fig:cupsed-hypo} we depict the computed approximation $\Gamma_n$
against the original curve $\Gamma$. Note that in the first two plots the fitting of the two curves
is not far from being perfect. Even in the cusped case, pictured in Figure~\ref{fig:cupsed-hypo}, the fitting is
remarkably close, despite the low degree of the moment matrix used.

\begin{figure}
\begin{center}
\includegraphics[scale=0.40]{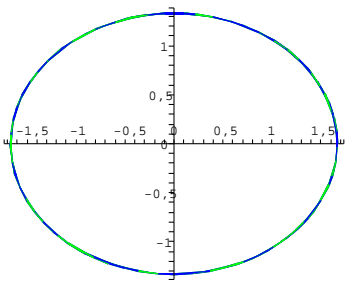}
\caption{Recovery of an  ellipse, with $n=3$}
\label{fig:ellipse}
\end{center}
\end{figure}

In Figure~\ref{fig:ellipse} we illustrate the reconstruction of an ellipse by using only the first 16 moments in
(\ref{eq:trun-mom}), i.e., by taking $n=3$.

\begin{figure}
\begin{center}
\includegraphics[width=0.35\linewidth]{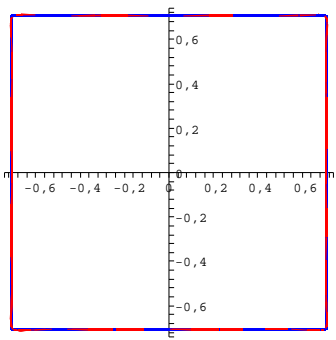}
\caption{Recovery of a  square, with $n=16$.}
\label{fig:square}
\end{center}
\end{figure}
In Figure~\ref{fig:square} we reconstruct a square by using the complex moments up to the degree $16$. We have chosen
$n=16$, so that the result can be compared with the recovery of a square, as shown on page 1067 of
\cite{GHMP} obtained using the \textit{Exponential Transform Algorithm} of the opus cited.
This is another
reconstruction algorithm based on moments. Of course, for concluding results regarding the comparison of the two
algorithms more experiments need to be contacted.

\begin{figure}
\begin{center}
\begin{minipage}[t]{1.5cm}
{\includegraphics[scale=0.40]{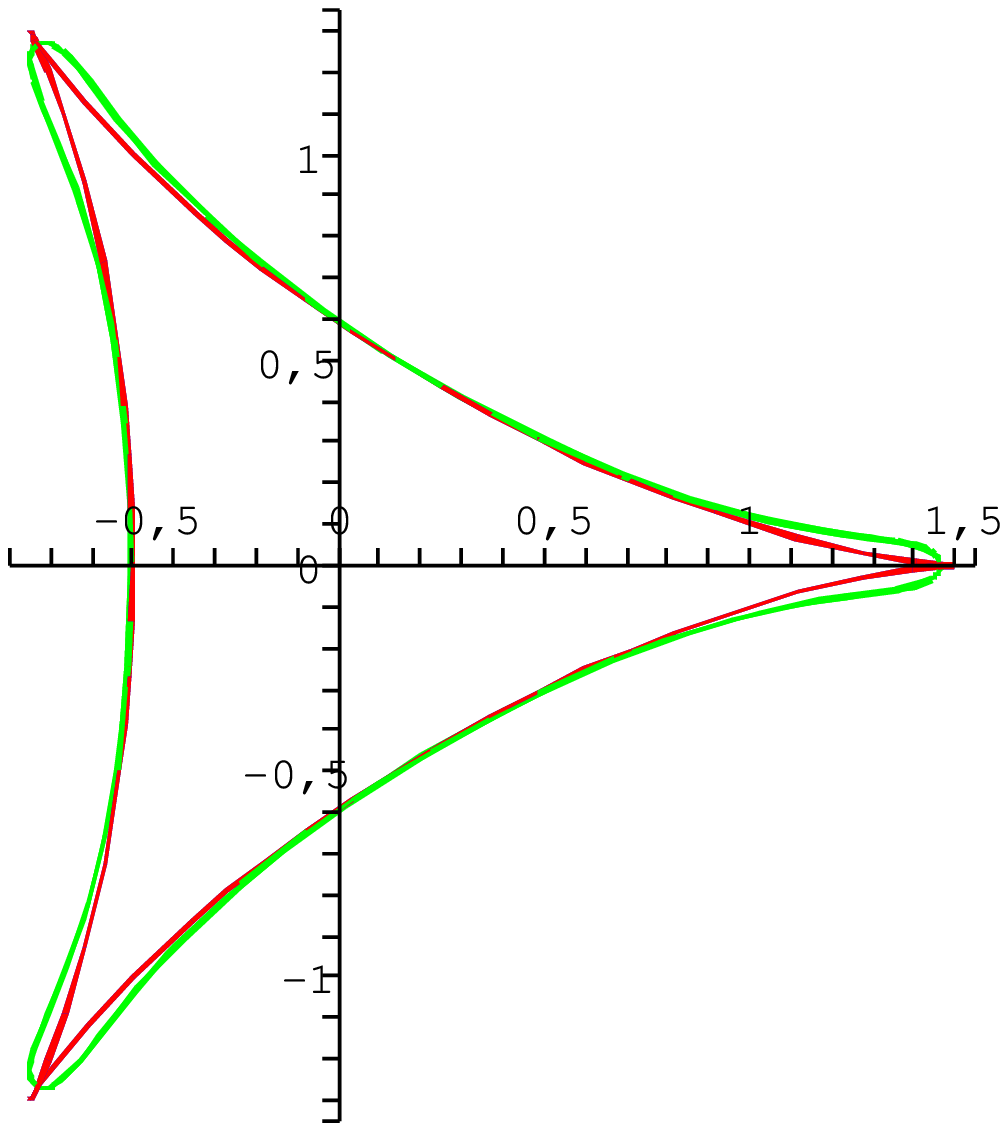}}
\end{minipage}
\qquad\qquad\qquad\qquad\qquad
\begin{minipage}[t]{1.5cm}
{\includegraphics[scale=0.40]{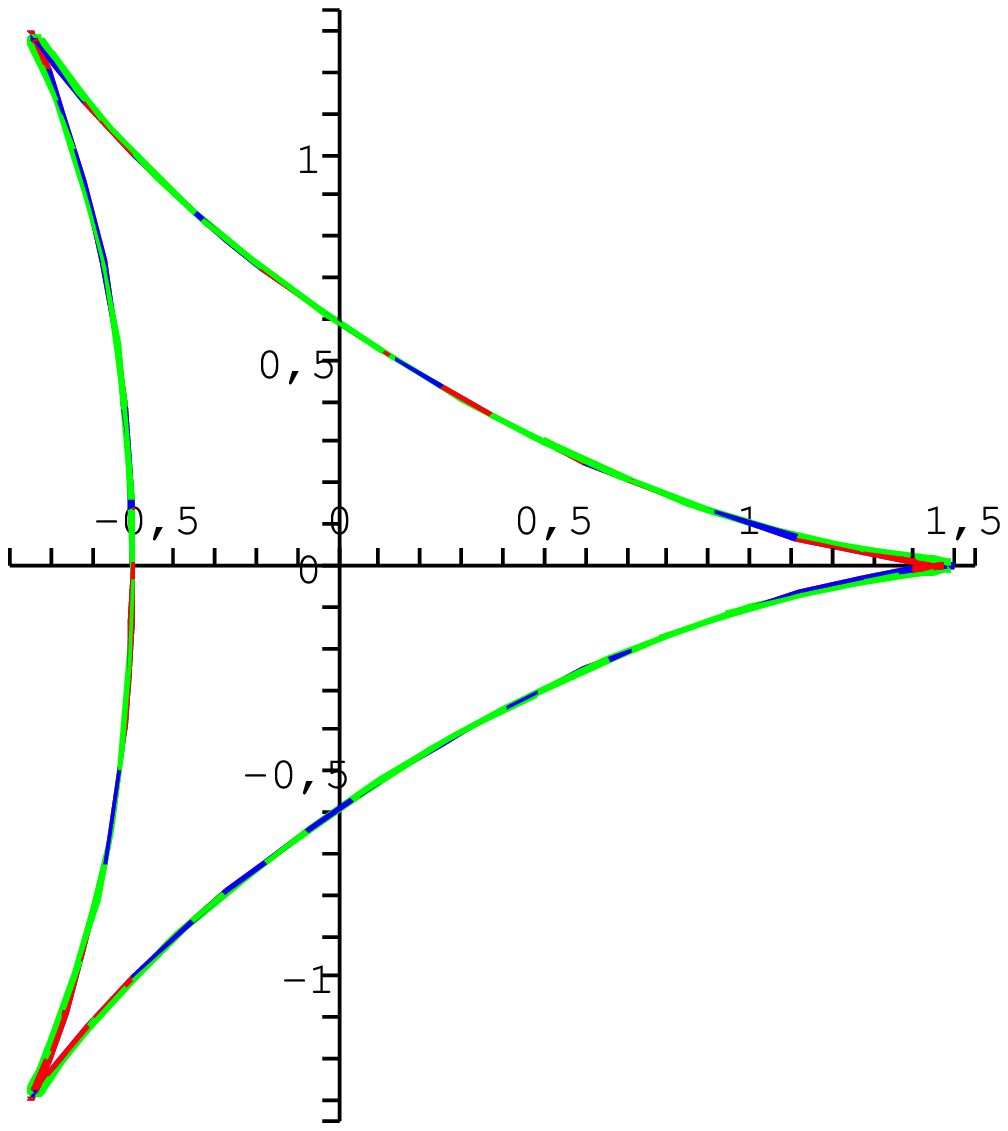}}
\end{minipage}
\caption{Recovery of a 3-cusped hypocycloid, with $n=10$ (left) and $n=20$ (right).}
\label{fig:cupsed-hypo}
\end{center}
\end{figure}

In order to show that our reconstruction algorithm works equally well for domains where the theory above does not apply,
we use it for the recovery of the boundary $H_3$ of the 3-cusped hypocycloid defined by (\ref{eq:Psi-hypom}) with $m=3$.
The application of the algorithm with $n=10$ and $n=20$ is depicted in Figure~\ref{fig:cupsed-hypo}.

Concluding, we note that the above algorithm is not suited for reconstructing unions of disjoint Jordan domains, in
contrast to the \textit{Archipelagos Reconstruction Algorithm} of \cite{GPSS}. At the other hand,
the simplicity of the construction and the proximity of the two curves $\Gamma_n$ and $\Gamma$, shown
in the figures, support that the proposed algorithm is more efficient when it comes to recovering single Jordan
domains.

\subsection{Coefficient estimates}\label{subsec:coef-esti}
We recall the expansion (\ref{eq:Psi}) of the inverse conformal mapping $\Psi:\Delta\to\Omega$ and note that $\Psi(w)/b$
belongs to the well-known class $\Sigma$ of univalent functions; see, e.g., \cite{Po75} and \cite{Du83}.

The following result settles, in a certain sense, the associated coefficient problem for an important subclass of
$\Sigma$. We refer to \cite[\S 4.9]{Du83} for a comprehensive discussion of the coefficient problem for other subclasses
of $\Sigma$.
\begin{theorem}\label{th:bn-decay}
Assume that $\Gamma$ is piecewise analytic without cusps and let $\omega\pi$, $0<\omega<2,$ denote its smallest
exterior angle. Then, there holds that
\begin{equation}\label{eq:bn-decay-1}
|b_n|\le c_1(\Gamma)\frac{1}{n^{1+\omega}}, \quad n\in\mathbb{N},
\end{equation}
and the order $1+\omega$ of $1/n$ is sharp in the sense that for certain $\omega$, there exists a Jordan curve $\Gamma$ of the same class, such that
\begin{equation}\label{eq:bn-decay-2}
|b_n|\ge c_2(\Gamma)\frac{1}{n^{1+\omega}}, \quad\text{for infinitely many } n.
\end{equation}
\end{theorem}
\begin{proof}
The estimate (\ref{eq:bn-decay-1}) can be established by means of the tools developed in Section~\ref{proofs-pa}.
More precisely, the following array of equations can be readily verified by using (\ref{eq:c_1-2}) and arguing
as in Section~\ref{sec:proof-thm:HnOmgae}:
\begin{align}\label{eq:nbn-esti}
-nb_n
&=\frac{1}{2\pi i}\,\int_{L_R}\Phi^{n}(z)dz=\frac{1}{2\pi i}\,\int_{\Gamma^\prime}\Phi^{n}(z)dz\nonumber\\
&=\frac{1}{2\pi i}\sum_j\int_{l_j^1\cup l_j^2}{\Phi^n(\zeta)}\,d\zeta
     +\frac{1}{2\pi i}\int_{\tau}{\Phi^n(\zeta)}\,d\zeta \nonumber\\
&=\frac{1}{2\pi i}\sum_j\int_{l_j^1\cup l_j^2}{\Phi^n(\zeta)}\,d\zeta
     +O(\rho^n),
\end{align}
for some $\rho:=\rho(\Gamma)<1$.

Hence, we only need to estimate the integral
$$
I_j^i:=\int_{l_j^i}{\Phi^n(\zeta)}\,d\zeta.
$$
This can be done by working as in deriving (\ref{eq:Iji}). Indeed, by using the estimate
\begin{equation*}
|\Phi(\zeta)|\le 1-cs^{1/\omega_j}<\exp(-cs^{1/\omega_j}),\quad \zeta\in l_j^i,
\end{equation*}
we obtain
\begin{equation*}
|I_j^i|\le c\int_0^\infty \textup{e}^{-cns^{1/\omega_j}}\,ds={c\,\omega_j}\Gamma(\omega_j)\,\frac{1}{n^{\omega_j}},
\end{equation*}
and the required result (\ref{eq:bn-decay-1}) follows from (\ref{eq:nbn-esti}), with $\omega:=\min_{j}\{\omega_j\}$.

An extremal domain, where (\ref{eq:bn-decay-2}) holds true, is provided by the case where $\Gamma$ is the canonical square
$\Pi_4$, with vertices at $1$, $i$, $-1$ and $-i$. In this case $\omega=3/2$, and by making use of the
rotational symmetry of $\Pi_4$ it is easily seen that the Schwarz-Christoffel formula for the normalized conformal map
$\Psi:\Delta\to\Omega$ takes the following expression:
\begin{align*}
\Psi(w)&=\text{cap}(\Pi_4)\int\left(1-\frac{1}{w}\right)^{\omega-1}\left(1-\frac{i}{w}\right)^{\omega-1}
\left(1+\frac{1}{w}\right)^{\omega-1}\left(1+\frac{i}{w}\right)^{\omega-1}dw\\
&=\text{cap}(\Pi_4)\int\left(1-\frac{1}{w^4}\right)^{\omega-1}dw,
\end{align*}
or, more explicitly,
$$
\Psi(w)=\text{cap}(\Pi_4)\left\{w+\sum_{k=1}^\infty(-1)^{k+1}\binom{a}{k}\frac{1}{4k-1}\frac{1}{w^{4k-1}}\right\},
$$
where $a:=\omega-1=1/2$, and $\binom{a}{k}$ denotes the binomial coefficient.
Hence, for $n=4k-l$, $k\in\mathbb{N}$ and $l\in\{0,1,2,3\}$, we have
\begin{equation}\label{eq:bn-square}
b_n= \left\{
\begin{array}{ll}
\text{cap}(\Pi_4)(-1)^{k+1}\binom{a}{k}\frac{1}{n}, &\text{if } l=1,\\
0, & \text{if } l\ne 1.
\end{array}
\right.
\end{equation}
Now, using the properties of the Gamma function $\Gamma(z)$, it is easy to verify that
$$
\binom{a}{k}=\frac{(-1)^k}{\Gamma(-a)}\frac{\Gamma(k-a)}{\Gamma(k+1)}
=\frac{(-1)^k}{\Gamma(-a)}\left\{\frac{1}{k^{1+a}}+O\left(\frac{1}{k^{2+a}}\right)\right\},
$$
and this, in conjunction with (\ref{eq:bn-square}), provides the required behaviour
\begin{equation}\label{eq:bn_equiv}
|b_n|\asymp\frac{1}{n^{1+\omega}},
\end{equation}
for $n=3,7,11,\ldots$.

Clearly, the above argument applies to any canonical polygon $\Pi_m$, with $m$-sides.
In particular, (\ref{eq:bn_equiv}) holds true for any $\Pi_m$, $m\ge 3$, with $\omega=(m+2)/m$ and
$n=km-1$, $k\in\mathbb{N}$. Thus, any $\Pi_m$ can serve as an extremal curve for the estimate
(\ref{eq:bn-decay-1}).
\end{proof}
We note that, since $\Psi(1)=1$, it is not difficult to obtain the following expression for the capacity of $\Pi_4$,
using the properties of hypergeometric functions:
\begin{equation}\label{eq:cap-square}
\text{cap}(\Pi_4)=\frac{\Gamma^2(1/4)}{4\pi^{3/2}}=0.834\,626\,841\,674\,072\cdots.
\end{equation}

\begin{remark}
In the case where $\Gamma$ is allowed to have cusps we recall, from Section 1,
the following estimate of Gaier \cite[\S 4.1]{Ga99}:
$$
|b_n|\le c(\Gamma)\frac{1}{n}, \quad n\in\mathbb{N}.
$$
This shows that the arguments on the first part of the proof of the theorem can be amended to cover the case
of zero exterior angles but not of angles of opening $2\pi$.
\end{remark}

\subsection{A connection with Operator Theory}\label{subsec:OT}
In a different reading, Theorem~\ref{th:bn-decay} brings in a connection with Operator Theory.
To testify this, we consider the Toeplitz  matrix $T_\Psi$ defined by the continuous extension of $\Psi(w)$ to the unit circle
$\mathbb{T}:=\{w:|w|=1\}$.  By this we mean the matrix
\begin{equation}\label{eq:Tmat}
{T_\Psi}:=\left[%
\begin{array}{cccccc}
   b_{0}& b_{1}& b_{2}& b_{3}& b_{4}&\cdots\\
   b& b_{0}& b_{1}& b_{2}& b_{3}&\cdots \\
    0    & b& b_{0}& b_{1}& b_{2}&\cdots \\
    0    &  0    & b& b_{0}& b_{1}&\cdots \\
    0    &  0    &   0   & b& b_{0}&\cdots \\
   \vdots& \vdots& \vdots& \ddots& \ddots&\ddots
\end{array}%
\right],
\end{equation}
defined by the coefficients of $\Psi(w)$ in its Laurent series expansion (\ref{eq:Psi}).
If the boundary $\Gamma$ is piecewise analytic without cusps, then Theorem~\ref{th:bn-decay} implies that
$\sum_{n=0}^\infty|b_n|<\infty$, and hence that the symbol
$\Psi$ of the Toeplitz matrix $T_\Psi$ belongs to the Wiener algebra; see, e.g., \cite[\S 1.2--1.5]{Bo-Gr05}.
This property leads to very interesting conclusions. For instance, to the conclusion that
$T_\Psi$ defines a bounded linear operator on the Hilbert space $l^2$ and that
\begin{equation}\label{eq:essspT}
\sigma_{ess}(T_\Psi)=\Gamma,
\end{equation}
where we use $\sigma_{ess}(L)$ to denote the  \textit{essential spectrum} of a bounded
linear operator $L$.

Consider  next the multiplication by $z$ operator $\mathcal{M}:f\to zf$ (also known as the
\textit{Bergman shift operator}), defined on the Hilbert space $L_a^2(G)$.
We note that $\mathcal{M}$ is a bounded linear operator on $L_a^2(G)$, such that
\begin{equation*}
\sigma_{ess}(\mathcal{M})=\Gamma;
\end{equation*}
see \cite{AxCoMcDo}. Hence, from (\ref{eq:essspT}) it follows that
\begin{equation}\label{eq:esssp=}
\sigma_{ess}(\mathcal{M})=\sigma_{ess}(T_\Psi).
\end{equation}
In a forthcoming paper \cite{SaSt2012}, we employ  results and tools from the present work to show that the
connection between the two operators $\mathcal{M}$ and $T_\Psi$ is much more substantial.

In order to emphasize the importance of the Bergman shift operator $\mathcal{M}$ in the theory of orthogonal polynomials,
we note that the proof of Theorem~\ref{thm:ftrr} relies heavily on the properties of $\mathcal{M}$; see \cite{PuSt} and
\cite{KhSt}.
Furthermore, we note that the stable Arnoldi GS process is based on the use of the polynomial
$zp_{n-1}$, i.e., on the application of $\mathcal{M}$ to $p_{n-1}$.

\subsection{The decay of the Bergman polynomials in $G$}
Here we refine the following estimate, which was derived in \cite[p.~530]{MSS} under the assumption that $\Gamma$ is piecewise analytic without cusps: For any compact subset $B$ of $G$ and for any $n\in\mathbb{N}$, it holds that
\begin{equation}\label{eq:pndecay-MSS}
|p_n(z)|\le\ c_1(\Gamma,B)\,\frac{1}{n^{s}},\quad z\in B,
\end{equation}
where
\begin{equation}\label{eq:pndecay-MSSs}
s:=\min_{1\le j\le N}\{\omega_j/(2-\omega_j)\}.
\end{equation}
(We use $c_j(\Gamma,B)$ to denote positive constants that depend only on $\Gamma$ and $B$.)
Note that $s\to 0$, if $\omega_j\to 0$, for some $j$, and hence for such cases, (\ref{eq:pndecay-MSS}) predicts a very slow decay for $p_n(z)$. The next theorem, however, shows that this decay cannot be slower than $O(1/\sqrt{n})$.
\begin{theorem}\label{lem:pnG}
Assume that $\Gamma$ is piecewise analytic without cusps. Then, for any $n\in\mathbb{N}$, it holds that
\begin{equation}\label{eq:pnG}
|p_n(z)|\le{c_2(\Gamma,B)}\,\frac{1}{n^\sigma},\quad z\in B,
\end{equation}
where $\sigma:=\max\{1/2,s\}$.
\end{theorem}

\begin{proof}
By using Cauchy's formula for the derivative in (\ref{eq:FnIntRep}) and by working as in the proof of
Theorem~\ref{th:bn-decay}, it is readily seen that
\begin{equation}\label{eq:GaierFn}
|F_{n+1}^\prime(z)|\le{c_3(\Gamma,B)}\,\frac{1}{n^\omega},\quad z\in B,
\end{equation}
where $\omega\pi$ ($0<\omega<2$) is the smallest exterior angle of $\Gamma$. This, in view of (\ref{eq:GnFn+1}), gives immediately
\begin{equation}\label{eq:Gnz}
|G_n(z)|\le {c_4(\Gamma,B)}\,\frac{1}{n^{1+\omega}},\quad z\in B.
\end{equation}

Next, since
$$
|q_{n-1}(z)|\le\frac{\|q_{n-1}\|_{L^2(G)}}{\sqrt{\pi}\,{\dist(z,\Gamma)}},\quad z\in G;
$$
see, e.g., \cite[p.\ 4]{Gabook87}, we obtain from Corollary~\ref{cor:qn-1L2} the estimate
\begin{equation}\label{eq:qnz}
|q_{n-1}(z)|\le {c_5(\Gamma,B)}\,\frac{1}{n},\quad z\in B.
\end{equation}

Finally, from (\ref{eq:qndef}), we have that
\begin{equation*}
|p_n(z)|\le\frac{\lambda_n}{\gamma^{n+1}}\left\{|G_n(z)|+|q_{n-1}(z)|\right\},
\end{equation*}
and this in view of (\ref{eq:lam/gam})--(\ref{eq:lam/gam-2}) and (\ref{eq:Gnz})--(\ref{eq:qnz}) yields
\begin{equation}\label{eq:pn-order-half}
|p_n(z)|\le c_6(\Gamma,B)\,\frac{1}{n^{1/2}},\quad z\in B.
\end{equation}
The result of the theorem follows by combining (\ref{eq:pndecay-MSS}) with (\ref{eq:pn-order-half}).
\end{proof}

\begin{remark}
Regarding sharpness of the exponent $\sigma$ of $n$ in (\ref{eq:pnG}), we recall the following result of
\cite[p.\ 531]{MSS}: \lq\lq If not all interior angles of $\Gamma$ are of the  form $\pi/m$, $m\in\mathbb{N}$, and if we
disregard in the definition of $s$ in (\ref{eq:pndecay-MSSs}) angles of this form, should there exists, then for any
$\varepsilon >0$, there is a
subsequence $\mathcal{N}_\varepsilon\subset\mathbb{N}$, such that for any $n\in\mathcal{N}_\varepsilon$:\rq\rq
$$
|p_n(z)|\ge c_7(\Gamma,B)\,\frac{1}{n^{s+1/2+\varepsilon}},\quad z\in B.
$$
\end{remark}

\bibliographystyle{amsplain}

\def\cprime{$'$}
\providecommand{\bysame}{\leavevmode\hbox to3em{\hrulefill}\thinspace}
\providecommand{\MR}{\relax\ifhmode\unskip\space\fi MR }
\providecommand{\MRhref}[2]{%
  \href{http://www.ams.org/mathscinet-getitem?mr=#1}{#2}
}
\providecommand{\href}[2]{#2}

\end{document}